\documentclass{article}
\usepackage[T1]{fontenc}
\usepackage[utf8]{inputenc}

\RequirePackage{amsthm,amsmath,amsfonts,amssymb}
\RequirePackage[numbers]{natbib}

\usepackage[margin=1.25in]{geometry}

\newcommand{\ac}[1]{\textcolor{red}{add citation}}
\usepackage{ltablex}
\usepackage{tabulary}
\usepackage{booktabs}
\usepackage{enumerate}
\usepackage{color}

\usepackage{color}
\usepackage{dsfont} 

\usepackage[most]{tcolorbox}

\usepackage{pgfplots}
\usepackage{subfigure}
\usetikzlibrary{patterns}
\usepackage{graphicx}
\usepackage{bm}

\usetikzlibrary{decorations.pathreplacing}

\usetikzlibrary{calc,patterns}

\usepackage{url}
\usepackage{fancyhdr}
\usepackage{indentfirst}

\usepackage{enumerate}
\usepackage[british]{babel}
\usepackage[colorlinks=true,citecolor=blue]{hyperref}
\usepackage{color}
\usepackage{comment}
\usepackage{tikz-cd}
\usepackage{bbm}
\usepackage{indentfirst}
\usepackage{authblk}
\pgfplotsset{compat=1.7}
\usepackage[colorlinks=true,citecolor=blue]{hyperref}
\usepackage{cleveref}
\usepackage{color}
\renewcommand{\cite}{\citet}
\usepackage{comment}
\usepackage{soul}

\usepackage{algorithm}

\renewcommand{\d}{\,\mathrm{d}}
\newcommand{\dd}{\overset{\mathrm{law}}{=}}
\newcommand{\p}{\mathbb{P}}
\newcommand{\q}{\mathbb{Q}}

\newcommand{\var}{\mathrm{Var}}   
\usepackage{xcolor}

\newcommand{\E}{\mathbb{E}}    
\newcommand{\R}{\mathbb{R}}    
\newcommand{\N}{\mathbb{N}}    

\theoremstyle{plain}
\newtheorem{theorem}{Theorem}
\newtheorem{corollary}[theorem]{Corollary}
\newtheorem{lemma}[theorem]{Lemma}
\newtheorem{proposition}[theorem]{Proposition}

\theoremstyle{definition}

\newtheorem{example}[theorem]{Example}

\theoremstyle{remark}
\newtheorem{remark}{Remark}

\usepackage{tikz}
\usetikzlibrary{arrows.meta,positioning}

\newcommand{\ee}{\varepsilon}

\newcommand{\n}[1]{\left\lVert#1\right\rVert}

\newcommand{\bz}{{\mathbf{z}}}
\newcommand{\bxi}{{\boldsymbol{\xi}}}

\newcommand{\bth}{{\boldsymbol{\theta}}}

\newcommand{\bb}{\mathbf{b}}
\newcommand{\bS}{\mathbb{S}}

\newcommand{\bx}{\mathbf{x}}
\newcommand{\by}{\mathbf{y}}

\newcommand{\bc}{\mathbf{c}}
\newcommand{\bce}{\hat{\mathbf{c}}_2\cdot\be}
\newcommand{\bla}{\boldsymbol\lambda}

\renewcommand{\bth}{\boldsymbol{\theta}}

\newcommand{\bet}{\boldsymbol{\eta}}
\usepackage{algorithm}
\usepackage{algpseudocode}
\def\d{\mathrm{d}}
\DeclareMathOperator\supp{supp}

\newcommand{\bone}{ {\mathbbm{1}} }
\renewcommand{\S}{\mathbb{S}}
\renewcommand{\bS}{\mathbf S}

\newcommand{\sF}{\mathscr{F}}

\usepackage{mathrsfs}

\newcommand{\be}{\mathbf{e}}
\newcommand{\z}{\mathbf{0}}

\renewcommand{\leq}{\leqslant}
\renewcommand{\epsilon}{\varepsilon}

\allowdisplaybreaks

\title{Large Deviations of First Passage Times of Branching Random Walks in \texorpdfstring{$\R^d$}{}: Asymptotics and Algorithms}
\author{
    Jose Blanchet\thanks{Department of Management Science and Engineering, Stanford University. Email: \texttt{jose.blanchet@stanford.edu}}, ~
    Wei Cai\thanks{Department of Mechanical Engineering, Stanford University. Email: \texttt{caiwei@stanford.edu}}, ~
    Shaswat Mohanty\thanks{Department of Mechanical Engineering, Stanford University. Email: \texttt{shaswatm@stanford.edu}}, ~
    Zhenyuan Zhang\thanks{Department of Mathematics, Stanford University. Email: \texttt{zzy@stanford.edu}}
}
\begin{document}

\maketitle
\begin{abstract}
We investigate the large deviation probabilities of first passage times (FPT) of discrete-time supercritical non-lattice branching random walks (BRWs) in $\mathbb{R}^d$ where $d\geq 1$. 
The FPT refers to the first time the BRW enters a ball of radius one with a distance $x$ from the origin, conditioned upon the process's survival. 
Furthermore, we apply the spine decomposition technique to construct an asymptotically optimal polynomial-time algorithm for computing the lower large deviation probabilities of the FPT. 
The accuracy of our algorithm is also verified numerically. 
Our analysis not only provides a deeper theoretical understanding of these stochastic processes but also offers new insights into the microstructural features that are key to characterizing the strength of polymers. 
%
\end{abstract}

\section{Introduction}

Branching random walks (BRWs) constitute a fundamental class of stochastic processes describing the spatial evolution of particles undergoing independent random movements interspersed with branching events. We focus on a discrete-time, supercritical, non-lattice BRW in $\mathbb{R}^d$, initiated by a single particle at the origin. Each particle independently produces offspring and then moves according to a common non-lattice jump distribution. A natural quantity of interest in this context is the first passage time (FPT), $\tau_x$, defined as the earliest time at which any particle reaches a target set, typically a unit ball around a distant point $x \in \mathbb{R}^d$. Understanding the probabilistic behavior of $\tau_x$, particularly in the regime of rare events, has significant theoretical implications and practical relevance in areas such as polymer physics and material science, as we will discuss.

This paper rigorously investigates the tail asymptotics of $\tau_x$, where the lower-tail (upper-tail) corresponds to the scenario of having unusually short (long) FPTs.  We further develop efficient rare-event simulation algorithms to estimate lower-tail probabilities. The style of algorithmic analysis that we present could potentially serve as a template to develop efficient Monte Carlo methods for upper-tail estimates, but this is not direct and will make the paper significantly longer. Since the simulation of upper-tail estimates are not directly applicable to the downstream tasks that we have in mind in the material science applications, we do not pursue the corresponding algorithms in this setting. Our main contributions in this paper are the following:
\begin{itemize}
    \item \textbf{Sharp lower-tail large deviations:} We derive asymptotics for the lower-tail probabilities $\p(\tau_x \le t)$ up to constants, corresponding intuitively to the rare scenario of a single unusually fast trajectory reaching the target set. See Theorem~\ref{thm:Ld1} in Section \ref{sec:LD}.

    \item \textbf{Upper-tail large deviations:} We establish logarithmic asymptotics for upper-tail probabilities $\p(\tau_x \ge t)$, capturing scenarios where reaching the target set is abnormally slow. These events arise due to a single particle drifting unfavorably before initiating a standard branching process, characterized via an optimization formulation described by Theorem \ref{thm:LD2} in Section \ref{sec:LD}.

    \item \textbf{Polynomial-time rare-event simulation algorithms:} We introduce novel simulation algorithms (Theorem \ref{prop:alg} and Corollary \ref{thm:alg}, Section \ref{sec:3.2}), based on established spine decomposition techniques. Remarkably, our algorithms achieve polynomial complexity with respect to the rarity parameter $x$, controlling both relative variance and computational effort simultaneously. This polynomial complexity stands in sharp contrast to brute-force Monte Carlo, which requires exponentially (in $x$) many replications for fixed relative precision, with each simulation incurring exponential computational cost (in $x$) due to the branching mechanism inherent in the BRW.

    \item \textbf{Insights into most-likely rare-event paths:} We provide detailed structural insights into trajectories associated with atypically short/long-FPT events, elaborated in the development of the technical proofs and mostly the algorithms in Section \ref{sec:alg}. Figure \ref{fig:fpt_extremal} illustrates some of the insights: the lower-tail trajectory heads in the direction of the target, whereas the upper-tail trajectory moves away from the target with a delayed branching event before heading towards it.
\end{itemize}

\begin{figure}[!htbp]
    \centering
     \vspace{-0.4cm}
     \includegraphics[width=0.9\linewidth]{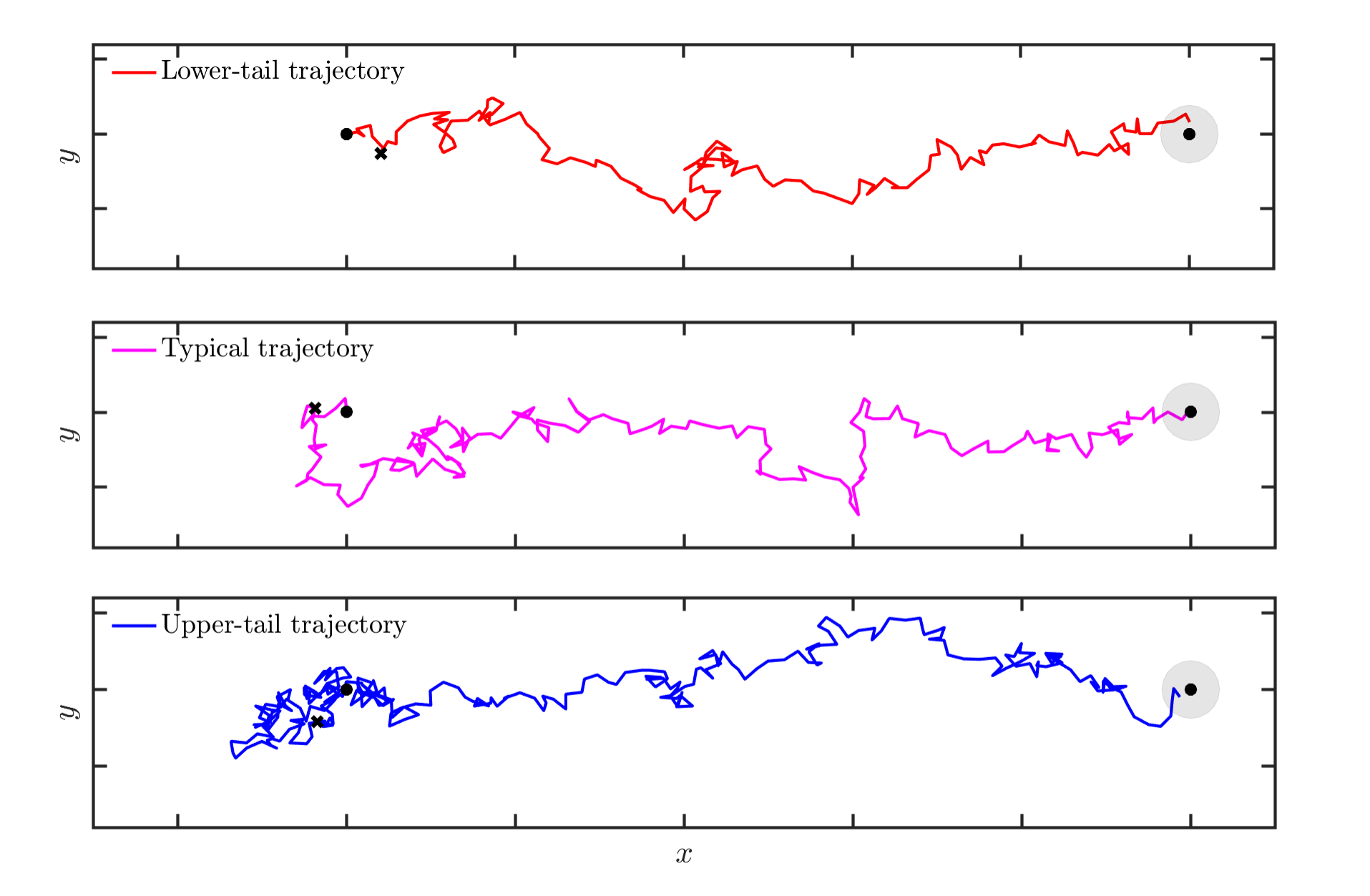}
    \vspace{-0.4cm}
    \caption{Trajectories that realize the FPT of a BRW from the lower tail, a typical realization (closer to the mean value), and the upper tail of the FPT distribution obtained through brute-force Monte Carlo simulations. The BRWs are simulated in dimension $d=3$ with jumps uniformly distributed in $\S^2$, and the plots show the projections in the first two coordinates. The black dots indicate the starting and ending locations of the trajectories, and the black crosses denote the sites of the first branching events). }
    \label{fig:fpt_extremal}
\end{figure}

The significance of our results can be highlighted in three complementary aspects:

\paragraph{Probability Theory Context.} While substantial literature addresses large deviations and extremal behavior for one-dimensional BRW (and its continuous counterpart, branching Brownian motion (BBM)) \citep{addario2009minima,aidekon2013branching,bramson2016convergence,buraczewski2019large,luo2025precise}, multi-dimensional settings pose distinct and substantial challenges. Previous works on multi-dimensional BRWs have mainly focused on maximal norms \citep{berestycki2024extremal,bezborodov2023maximal,mallein2015maximal}, with fewer studies dedicated explicitly to first passage times   \citep{blanchet2024first,blanchet2024tightness,zhang2024modeling}. Our analysis makes significant progress in understanding these multi-dimensional scenarios, highlighting fundamental differences from the well-explored one-dimensional case.

\paragraph{Algorithmic Innovation.} To the best of our knowledge, our polynomial-time algorithms represent the first substantial breakthrough regarding computational complexity for rare-event simulation of BRWs. Existing literature predominantly requires exponential computational resources or deals with certain simplified settings (such as BRW with negative drift) \citep{basrak2022importance,brunet2020generate,conroy2022efficient}. Our innovative use of spine decomposition—a powerful importance sampling analogy—significantly advances practical simulations in complex multidimensional BRWs. Our methodology also enables polynomial-time estimation of critical constants involved in upper-tail large deviation probabilities for one-dimensional BRW maxima, providing practical computational tools for results previously considered challenging to implement \citep{luo2025precise}. In our paper, we focus on polynomial-type complexity of the lower tail of the FPT distribution because those events are more relevant from the standpoint of the applications that are of interest to us, namely, the study of most likely configurations that lead to early fracture or damage in elastomers. We explain this connection in the next paragraph.

\paragraph{Applications to Polymer Networks.} Our investigation is strongly motivated by applications to polymer physics, particularly in modeling elastomers using coarse-grained molecular dynamics (CGMD) simulations \citep{mohanty2025strength,yin2020topological,yin2024network,zhang2024modeling}. In polymer networks, the distribution of shortest paths (SP) between distant cross-links significantly influences macroscopic mechanical responses, including stress-strain behavior and fracture mechanisms. Empirical evidence from CGMD simulations demonstrates that polymer network fracture mechanisms are governed by these shortest paths. Shorter-than-typical first-passage events in the BRW setting directly translate to polymer networks characterized by unusually straight load-bearing paths, making them prone to damage and fracture under strain. This is why our algorithms focused on lower-tail events. Thus, our methods facilitate predictive simulations, which are key to understanding and optimizing the strength and durability of polymeric materials.
Applying insights from these BRW simulations to accelerating CGMD simulations is a promising direction for future research.\\

By combining rigorous probabilistic analysis, innovative computational methods, and clear connections to the study of polymer networks, our work offers comprehensive insights and powerful new tools for both theoretical and algorithmic analysis of BRWs and their applications.

The rest of the paper is organized as follows. In Section~\ref{sec:LD}, after introducing the main assumptions in our paper, we present our theoretical large deviation results  (upper and lower tails). Section~\ref{Sub:spine_MLP} motivates the structure of the ``most likely path'' and the so-called spine decomposition in the context of a highly simplified model involving exponentially many independent random walks. The goal of this simplified model is to motivate the nature of the importance sampling strategy eventually used in our algorithms and the pruning of paths that will be necessary to control computational costs. Then, we move on to the rigorous development of our technical proofs in Section~\ref{sec:technical_proofs}. The construction of the importance sampling strategy is provided in Section~\ref{sec:alg}, together with a precise discussion of asymptotic optimality in the context of rare-event simulation algorithms. We also provide a thorough computational complexity analysis leading to a procedure that controls the relative mean squared error of the estimator with polynomial complexity in $x$. Numerical experiments verifying the intuition used in our proof and the empirical performance of our procedure are given in Section \ref{sec:numerical}.
Section~\ref{sec:conclusions} provides some conclusive remarks.

\paragraph{Notation.} Throughout, we consider a dimension $d\in\N$ that is fixed. $B_r(\bx)$ denotes the ball of radius $r$ centered at $\bx\in\R^d$. Typically, in this paper, a bold symbol refers to a vector.  We write $B_\bx=B_1(\bx)$. If $\bx=(x,0,\dots,0)$, we also write $B_x=B_\bx$. We use $\z$ to denote a vector of zeros when the dimension is clear and write $\be=(1,\z)\in\R^d$. Let $[n]=\{1,2,\dots,n\}$. We write $A\ll B$ (or $B\gg A$, or $A=O(B)$) if there exists a constant $C>0$ possibly depending on the law of the BRW such that $A\leq CB$. We also write $A\asymp B$ if $A\ll B\ll A$. We write $A(x)\propto B(x)$ if there exists some constant $C>0$ such that for all $x$, $A(x)=CB(x)$. Denote by $\delta_\bx$ the Dirac delta mass at $\bx$.
For a real-valued function $f$ we denote by $\mathrm{ran}f$ the range of $f$, and by $(\mathrm{ran}f)^\circ$ the interior of the range of $f$. The notation $\n{\cdot}$ stands for the Euclidean norm. 
Other frequently used notation are collected in Appendix \ref{appendix:notation}.

\section{Large deviations for first passage times of BRW}\label{sec:LD}

\subsection{Setup and main results}

Let us start with a formal definition of the branching random walk (BRW). We consider a discrete-time BRW model with an offspring distribution $\zeta$, satisfying $\p(\zeta=i)=p_i,~i\geq 0$. This means that at each time step, each particle independently gets replaced by $i$ particles at the same location with probability $p_i$.  The case of $i=0$ corresponds to the particle being terminated. Let $\rho=\E[\zeta]$ be the mean of the offspring distribution (which is finite under our assumptions below). 
The $d$-dimensional random walk increment is given by $\bxi$, and the first particle starts from the origin $\z\in\R^d$. 
Let $\phi_\bxi$ denote the moment generating function (MGF) of $\bxi$ and
\begin{align}
    I(\bx):=\sup_{\bla\in\R^d}\Big(\bla\cdot\bx-\log\phi_\bxi(\bla)\Big)=\sup_{\bla\in\R^d}\Big(\bla\cdot\bx-\log\E[e^{\bla\cdot\bxi}]\Big)\label{eq:I long}
\end{align}
denote the large deviation rate function for $\bxi$. 

In this paper, we consider the following assumptions on the BRW process:
 \begin{itemize}
     \item [(A1)] the offspring distribution $\zeta$ has a finite second moment, i.e., $\sum_i i^{2}\,p_i<\infty$; \label{A1}
     \item [(A2)] the law of $\bxi$ is integrable and centered, i.e., $\E[\bxi]=\z$;
         \item [(A3)] $\log\rho\in(\mathrm{ran}I(\cdot,\z))^\circ$, where $I(\cdot,\z)$ refers to the function $I$ with the last $d-1$ variables fixed zero; in addition, the MGF $\phi_\bxi$ is well-defined on all of $\R^d$;
\item [(A4)] the law of $\bxi$ is non-lattice in the sense that for all $\bx\in\R^{d}\setminus\{\z\}$, $|\E[e^{\mathrm{i}\bx\cdot\bxi}]|< 1$;
\item[(A5)] $p_0+p_1>0$.
 \end{itemize}

 Under conditions (A2) and (A4), $I(\bx)$ is continuously differentiable in the interior of its domain, continuous on its domain, and attains a unique minimum at $\bx=\z$ with $I(\z)=0$. By (A3), there exists $c_1>0$ such that 
 \begin{align*}
     I(c_1,\z)=\log\rho.
 \end{align*}
 The constant $c_1$ can be understood as the linear speed of the typical BRW frontier along the direction of the first coordinate.

 Let $V_n$ denote the collection (i.e. set) of particles at the time step $n$. Denote by $S:=\{\forall n\geq 1,~|V_n|>0\}$ the {survival event} that the underlying branching process survives at all times. Let $q=1-\p(S)$ be the extinction probability. 

    


  For $x\in\R$, we let $B_x$ denote a ball of radius one centered at $\bx=(x,0,\dots,0)$ in $\R^d$. For $v\in V_n$, we let $\bet_v$ denote the location of the particle $v$. We define the first passage time (FPT) $\tau_x$  of the BRW to $B_x$, that is,
$$\tau_x:=\min\{n\geq 0 : \exists\, v\in V_n,~ \bet_{v}\in B_x\}.$$
Note that our results in this paper can be easily generalized to target balls centered at $x$ with a fixed radius $r>0$ (that does not depend on $x$). 
The following are our main results.

\begin{theorem}[lower-tail deviation]\label{thm:Ld1}
    Assume that (A1)--(A4) hold. For $\hat{c}_1>c_1$, we have
\begin{align}
    \p\left(\left.\tau_x<\frac{x}{\hat{c}_1}\right| S\right)\asymp x^{-d/2}
    \exp\left(-\frac{x}{\hat{c}_1}\big(I(\hat{c}_1,\z)-\log\rho\big)\right).\label{eq:asymp lb}
\end{align}
    In particular,
    $$\lim_{x\to\infty}-\frac{1}{x}\log\p\left(\left.\tau_x<\frac{x}{\hat{c}_1}\right| S\right)=\frac{1}{\hat{c}_1}\Big(I\big(\hat{c}_1,\z\big)-\log\rho\Big).$$
\end{theorem}

The condition of $\hat{c}_1>c_1$ means that the BRW frontier along the first coordinate expands at a mean speed $\hat{c}_1$ that is higher than the typical value $c_1$.  In other words, the trajectory of the particle achieving the FPT is straighter than expected from a typical scenario.
See Figure \ref{fig:nonsphere2} for an illustration. 
The asymptotics \eqref{eq:asymp lb}, especially the tightness of the lower bound, will become crucial later in Section \ref{sec:alg} in showing the polynomial complexity of the estimators of the large deviation event. The underlying intuition of the lower-tail deviation event $\{\tau_x<x/\hat{c}_1\}$ is that at time $n\approx x/\hat{c}_1$, there exists a single ``fast'' trajectory (among roughly $\rho^n$ ones) that reaches the target $B_x$, while the other trajectories stay close to normal except for their parts that overlap with the fast trajectory.

\begin{figure}[!htbp]
    \centering
    \includegraphics[width=0.65\textwidth]{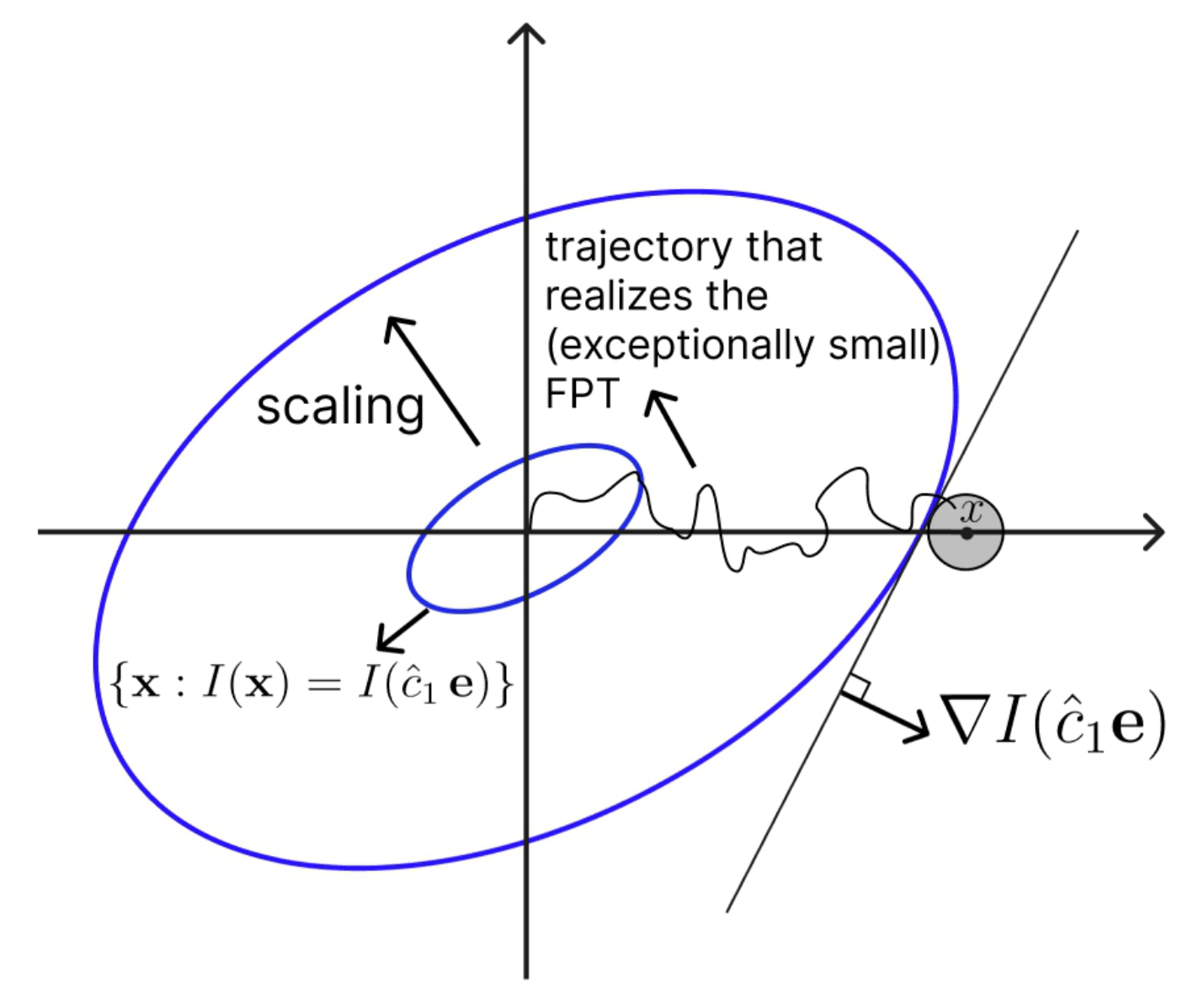}
    \caption{Illustration of the lower deviation of $\tau_x$. The trajectory that realizes the (exceptionally small) FPT roughly resembles an exponentially tilted random walk with mean $\hat{c}_1\be$.}
    \label{fig:nonsphere2}
\end{figure}
   
We conjecture that there exists a constant $C>0$ depending on the law of the BRW and $\hat{c}_1$ such that the lower deviation probability has the following precise asymptotics
$$\p\left(\left.\tau_x<\frac{x}{\hat{c}_1}\right| S\right)=(C+o(1)) \, x^{-d/2} \, \exp\left(-\frac{x}{\hat{c}_1}\big(I(\hat{c}_1,\z)-\log\rho\big)\right).$$
The analogous result for the maximum of a one-dimensional BRW has recently been established in \citep[Theorem 1.3]{luo2025precise}.

The analysis of the upper deviation probabilities of $\tau_x$ involves the following optimization problem: given $\bx=(x,\z)=x\be,\,\rho,\,\gamma=-\log\E[\zeta  q^{\zeta -1}]\in(0,\infty]$, and $\hat{c}_1<c_1$,
\begin{align}\begin{split}
    \text{minimize}\quad &\gamma\alpha+\alpha I\Big(\frac{\by_x}{\alpha}\Big),\\
    \text{subject to}\quad &(\alpha,\by_x)\in(0,1/\hat{c}_1)\times\R^d:~I\Big(\frac{\be-\by_x}{1/\hat{c}_1-\alpha}\Big)\geq\log \rho.
\end{split}\label{eq:max}
\end{align}
Denote by $T(x,\hat{c}_1;\rho,\gamma)$ the optimal value of \eqref{eq:max}. 

\begin{theorem}[upper-tail deviation]\label{thm:LD2}
    Assume (A1)--(A5), then for $\hat{c}_1<c_1$, 
    $$\lim_{x\to\infty}-\frac{1}{x}\log\p\Big(\tau_x>\frac{x}{\hat{c}_1}\mid S\Big)=T(x,\hat{c}_1;\rho,\gamma).$$ 
\end{theorem}

The intuition for the upper deviation is slightly more involved than the lower deviation and is shown in Figure \ref{fig:nonsphere}: at some time $\alpha x$, there remains a single particle and it is located near the point $x\by_x$; after time $\alpha x$, the particle performs a standard BRW until one of its descendants reaches $B_x$. The same phenomenon for the large deviation of the maximum of one-dimensional spatial branching processes appears in \citep{chen2020lower,chen2020branching,gantert2018large}. 

\begin{figure}[!htbp]
    \centering
    \includegraphics[width=0.65\textwidth]{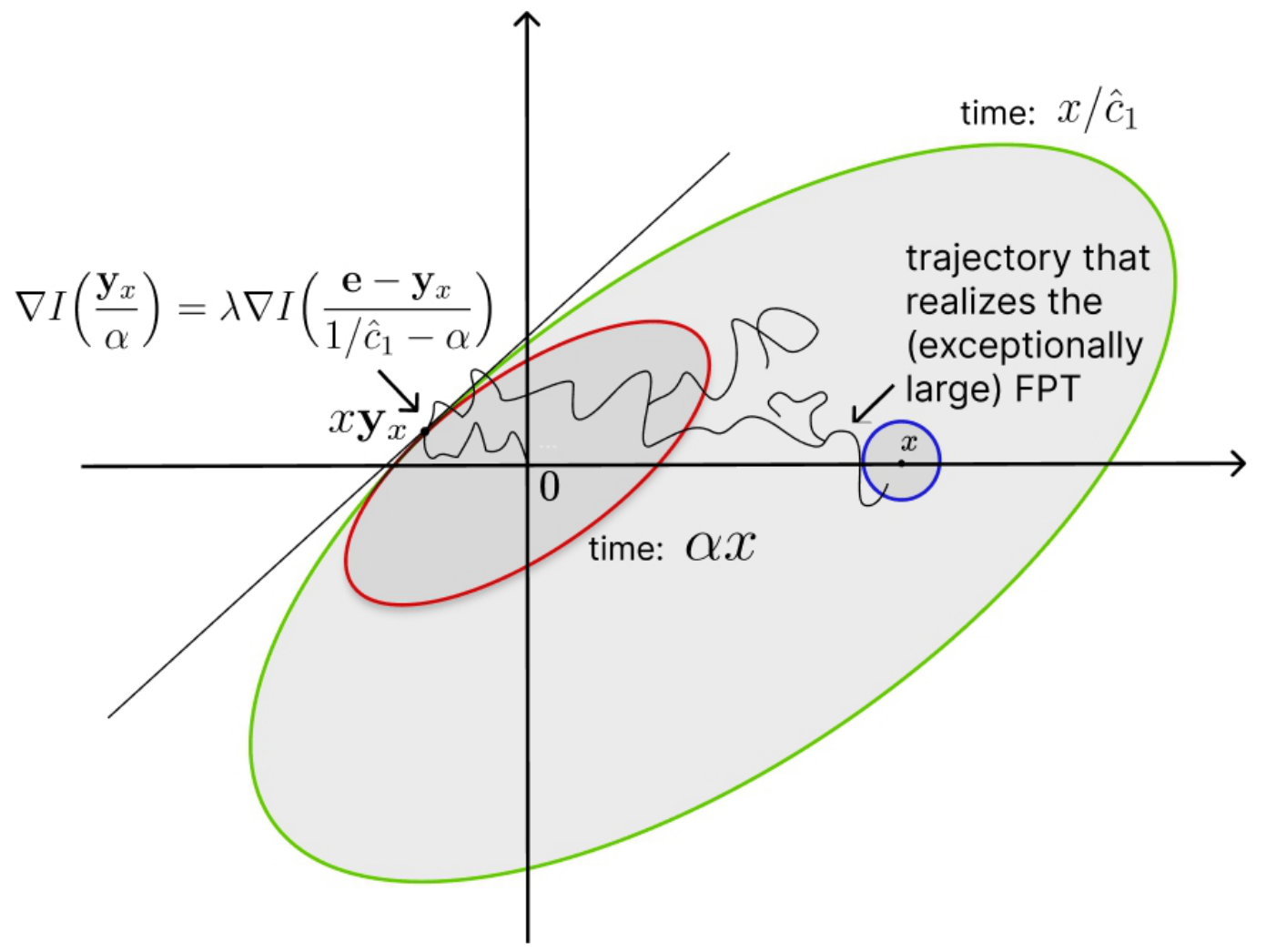}
    \caption{Illustration of the upper deviation of $\tau_x$. The first particle persists for a time $\alpha x$ without branching, landing at the point $x\by_x$. After it branches, the system performs an ordinary BRW. The green and red ellipses showcase the scaled level sets of the large deviation rate functions. Intuitively speaking, the ellipses share a common tangent line at $x\by_x$, with the red ellipse lying inside the green one. }
    \label{fig:nonsphere}
\end{figure}

\begin{example}
Suppose that $\bxi$ is spherically symmetric. By definition, $I(\bx)$ depends only on $\n{\bx}$ and is a strictly increasing function in $\n{\bx}$, say, given by $I(\bx)=I_{\n{\cdot}}(\n{\bx})$ for some function $I_{\n{\cdot}}:\R_+\to\R_+$. Given $\alpha\in[0,1/\hat{c}_1]$, the problem \eqref{eq:max} is equivalent to minimizing $\n{\by_x}$ given $\n{\be-\by_x}$. Therefore, 
%
$\by_x\in\{x\be, x \in \R_-\}$. 
The minimum value to \eqref{eq:max} then writes
$$\inf_{\alpha\in[0,1/\hat{c}_1]}\bigg(\gamma\alpha+\alpha I_{\n{\cdot}}\Big(\frac{1-(1/\hat{c}_1-\alpha)c_1}{\alpha}\Big)\bigg).$$
In other words, the single particle that remains at time $\alpha x$ tends to have a negative drift during the time interval $[0,\alpha x]$. 
 \end{example}

 \begin{example}
     Let $d=1$ and let $M_n$ denote the maximum among all particles' locations at time $n$. One naturally expects that $\p(\tau_x>x/\hat{c}_1)\approx \p(M_{x/\hat{c}_1}<x)$. By Theorem 3.2 of \citep{gantert2018large}, for $\hat{c}_1<c_1$,
     \begin{align*}
         -\frac{1}{x}\log\p(M_{x/\hat{c}_1}<x)&\to \frac{1}{\hat{c}_1}\inf_{t\in(0,\min\{1-\hat{c}_1/c_1,1\}]}\Big(t\gamma+tI(t^{-1}(\hat{c}_1-(1-t)c_1))\Big)\\
&=\inf_{t\in(0,1/\hat{c}_1-1/c_1]}\Big(t\gamma+tI(\hat{c}_1^{-1}t^{-1}(\hat{c}_1-(1-t\hat{c}_1)c_1))\Big)\\
&=\inf_{t\in(0,1/\hat{c}_1-1/c_1]}\left(t\gamma+tI\left(\frac{1-(1/\hat{c}_1-t)c_1}{t}\right)\right).
     \end{align*}
     On the other hand, for a fixed $\alpha$, the only feasible $\by_x$ in the problem \eqref{eq:max} is $\by_x=1-(1/\hat{c}_1-\alpha)c_1$. Minimizing in $\alpha$ shows that \eqref{eq:max} aligns with Theorem 3.2 of \citep{gantert2018large}.
 \end{example}

\subsection{The spine decomposition and the nature of the ``most likely path''}\label{Sub:spine_MLP}

In this section, we will provide a rough description of the most likely path towards the occurrence of the event $\{\tau_x \leq x/\hat{c}_1\} $ where $\hat{c}_1>c_1$, so that we can motivate both our algorithmic constructions in Section \ref{sec:alg} and its proof. Let us first assume $d=1$ and focus on a related (but greatly simplified) process, which is often used to motivate results for BRW. The same model often appears in the analysis of the BRW in various contexts, such as the maximum \citep{zeitouni2016branching} and large deviation rates \citep{gantert2018large}.

Let us write $\bar{M}_n := \max_{1 \leq k \leq \rho^n} S_n(k)$ , where $\{S_\cdot(k): k \geq 1\}$ is a sequence of i.i.d. random walks satisfying assumptions (A2)--(A4) that we imposed. For simplicity, we assume that $\rho\in\N$. In this highly simplified setting which ignores the BRW correlations, it is easy to see by inclusion-exclusion, that for $\hat{c}_1>c_1$,
\begin{equation}\label{eq:ind_mot}
\p(\bar M_{\lfloor x/\hat{c}_1\rfloor }  \geq x )= \rho^{\lfloor x/\hat c_1\rfloor}\p(S_{\lfloor x/\hat c_1\rfloor}(1)\geq x)(1+o(1)).   
\end{equation}
The logarithmic asymptotics for the upper tail of $\bar M_{\lfloor x/\hat c_1\rfloor}$ can be easily seen to coincide with those described in Theorem 3.2 of \citep{gantert2018large}: for $\hat{c}_1>c_1$, we have
     $$\lim_{n\to\infty}-\frac{1}{n}\log\p(\bar{M}_n>\hat{c}_1n\mid S)=\lim_{n\to\infty}-\frac{1}{n}\log\p(M_n>\hat{c}_1n\mid S)=I(\hat{c}_1)-\log\rho.$$

A natural importance sampling strategy to consider for the event $\{\bar M_{\lfloor x/\hat{c}_1\rfloor }  \geq x\}$, and one that describes the most likely rare-event path in this case, is to first select $1 \leq k \leq \rho^{\lfloor x/\hat c_1\rfloor}$ uniformly at random and then bias $S_{\lfloor x/\hat c_1\rfloor}(k)$ using exponential tilt with the parameter $\hat c_2 $ that induces the mean $\hat c_1$; namely, $\hat c_2 = I'(\hat c_1)$. Then, one could sample the remaining random walks under the nominal distribution. If $\q$ is the measure induced by this sampling strategy, it follows that the likelihood ratio $Q_x:=\d\q/\d\p$ satisfies 
\begin{align*}
Q_x:=\frac{\d\q}{\d\p}= \rho^{-\lfloor x/\hat c_1\rfloor}\sum_{1\leq k\leq \rho^{\lfloor x/\hat c_1\rfloor}}\exp\left(\hat c_2 S_{\lfloor x/\hat c_1\rfloor}(k)-\lfloor x/\hat c_1\rfloor\log\phi_\bxi(\hat c_2)\right).
\end{align*}
The importance sampling estimator for this problem is $Z_x:=Q^{-1}_x\bone_{\{\bar M_{\lfloor x/\hat c_1\rfloor}\geq x\}}$. Note that $Z_x$ is an unbiased estimator of $\p(\bar M_{\lfloor x/\hat{c}_1\rfloor }  \geq x )$. Thus, its variance is given by 
\begin{align*}
\E^{\q}[Z_x^2] =\E[Q^{-1}_x\bone_{\{\bar M_{\lfloor x/\hat c_1\rfloor}\geq x\}}] \leq \sum_{k=1}^{\rho^{\lfloor x/\hat c_1\rfloor}}\E[Q_x^{-1}\bone_{\{S_{\lfloor x/\hat c_1\rfloor}(k)\geq x\}}],
\end{align*}
where $\E^{\q}$ denotes the expectation operator under $\q$ and $\E$ is the expectation under $\p$. Since $Q^{-1}_x \leq \rho^{\lfloor x/\hat c_1\rfloor} \exp(-\hat c_1 S_{\lfloor x/\hat c_1\rfloor}(k)+\lfloor x/\hat c_1\rfloor\log\phi_\bxi(\hat c_2))$  for all $k$ (just ignoring the remaining positive terms in the denominator of $Z_x$), using the constraint that $S_{\lfloor x/\hat c_1\rfloor}(k)\geq x$ we conclude that 
\begin{align}\label{eq:ind_2nd_bnd}
\E^{\q}[Z_x^2]  \leq \rho^{2\lfloor x/\hat c_1\rfloor}\exp(-\lfloor x/\hat c_1\rfloor I(\hat c_1))\p(S_{\lfloor x/\hat c_1\rfloor}(k)\geq x) = O\left(\p(\bar M_{\lfloor x/\hat c_1\rfloor}\geq x)^2\sqrt{x}\right).
\end{align}
The previous bound uses \eqref{eq:ind_mot} and precise large deviations estimates of a random walk due to the Bahadur--Rao theorem (Theorem 3.7.4 of \citep{dembo2009large} and Remark (c) that follows). The conclusion is that the coefficient of variation of this estimator (i.e., standard deviation divided by the probability of interest) grows like $x^{1/4}$. So, to obtain, say, a 95\% confidence interval of size $\epsilon$ in relative terms, it suffices (using, for example, Bernstein's inequality; see Corollary \ref{thm:alg}) to simulate $O(\sqrt x/\epsilon^2)$ i.i.d. replicates of the estimator $Z_x$ and average those replicates. 

There is an important issue one would have to deal with, namely, the fact that the estimator $Z_x$ takes an exponential amount of time to produce because one would need to produce $\rho^{\lfloor x/\hat c_1\rfloor}$ many random walks. Of course, most of the contribution arises only from the ``special'' random walk chosen initially under $\q$. 
In order to deal with this exponential complexity, one needs to introduce a small bias, which needs to be controlled (and reduced if necessary) as a function of the relative error tolerance parameter $\epsilon$. 

Although the previous discussion provides a highly simplified caricature of the BRW problem, there are parts that are analogous to the BRW problem. First, the role of the ``special'' random walk. In the more interesting setting of the BRW, the analogous description would correspond to selecting uniformly at random, at time $x/\hat c_1$, one of the leaves of the tree generated at that time (the fact that such a set of leaves is random creates a complication relative to the simplified case studied earlier). Then, we can proceed with a biasing strategy using exactly the same exponential tilting parameter (i.e. $\hat c_2$) used earlier, this time applied to the random walk associated with such chosen leaf (tilting each increment along that branch of the tree). In the BRW, we can think of the biased random walk as a ``spine'' from which branches of the BRW will emanate. The rest of the BRW should be sampled under the nominal dynamics; except for the immediate generations attached to the spine, which have a bias connected to the random set of leaves mentioned earlier.

There are several ways in which one can deal with the issue of having a random number of leaves in generation $n$, as opposed to $\rho^n$, as in our simplified discussion. The standard approach, which we adopt, is to size-bias the distribution of children along the spine (which affects the immediate generations attached to the spine as indicated above). This allows us to produce the spine forward while maintaining a very similar likelihood ratio as the direct strategy inspired by the simplified example. 

Another obstacle that we need to overcome is an asymptotic lower bound estimate for the probability of interest, similar to the one that is implicitly used to conclude \eqref{eq:ind_2nd_bnd}; there we took advantage of the independence of $\{S_{\cdot}(k):k\geq 1\}$ and invoked the Bahadur--Rao estimates, while in our BRW context we need to develop the required estimate.

Finally, the last and more delicate obstacle in our problem involves controlling the bias that we must incur in order to keep the estimator within polynomial complexity in $x$. This is substantially more complicated in our setting because along the spine and especially those within $O(\log x)$ units of the time horizon contribute significantly to the tail event of interest. This part requires isolating various events that will be discussed in Section \ref{sec:3.2}.

\subsection{Proof of the main results}\label{sec:technical_proofs}

\subsubsection{The spine decomposition}\label{sec:spineal}

Recall that $\{\bet_u:u\in V_n\}$ denotes the locations of the particles in generation $n$. For $\bth\in\R^d$ such that $\E[e^{\bth\cdot\bxi}]<\infty$, let
\begin{align}
    \psi(\bth):=\log\E\bigg[\sum_{u\in V_1}e^{\bth\cdot \bet_u}\bigg]=\log\rho+\log\phi_\bxi(\bth).\label{eq:psi}
\end{align}
Denote by $\sF_n:=\sigma(\bet_u:u\in V_{k},\,k\in[n])$ the natural filtration generated by the BRW process until time $n$. 
If $\psi(\bth)<\infty$, we define a new probability measure $\q_{\bth}$ (abbreviated as $\q$ in the following) by
$$\q(A)=\q_{\bth}(A):=\E[W_n^{\bth}\bone_A]\quad\text{ for }\quad n\geq 0,~A\in\sF_n,$$
where the Radon--Nikodym derivative
\begin{align}
  \frac{\d\q}{\d\p}=  W_n^{\bth}:=\sum_{u\in V_n}e^{\bth\cdot\bet_u-n\psi(\bth)}\label{eq:dq/dp}
\end{align}
forms a positive martingale under the original law $\p$. 
The interesting cases are $\bth=\hat{\bc}_2=\nabla I(\hat{c}_1,\z)$, where $\hat{c}_1>c_1$. In this case, we record for future use that by \eqref{eq:I long} and \eqref{eq:psi}, 
\begin{align}\psi(\hat{\bc}_2)
=\hat{c}_1\bce+\log\rho-I(\hat{c}_1)\label{eq:psibth}
\end{align}
and that by the continuous differentiability of $I$ on its domain, there exist $\bar{c}_1\in(c_1,\hat{c}_1)$ and $\ee_1>0$ such that
\begin{align}
    I(\hat{c}_1)-\log\rho>(\hat{c}_1-\bar{c}_1+2\ee_1)\bce.\label{eq:ccc}
\end{align}


The spine decomposition is an alternative representation of the law $\q_{\bth}$, introduced by \citep{lyons1997simple}. See also \citep{harris2017many} for a general version. We define a \emph{BRW with spine} as follows. Let $v_0$ be the unique particle at time $0$. We define a \emph{spine} $\{v_n\}_{n\geq 0}$ by selecting inductively a child $v=v_{n+1}$ of $v_n$ with probability proportional to $e^{\bth\cdot \bet_v}$ (or equivalently $e^{\bth\cdot (\bet_v-\bet_{v_n})}$). We use the pair $(N,\{\bxi_n\}_{1\leq n\leq N})$ to denote the number of descendants and their displacements of a generic particle. The offspring and displacements of particles not on the spine follow the nominal distribution (i.e. the same distribution as the original BRW: $N\dd\zeta$ and $\{\bxi_n\}$ are i.i.d.~with law $\bxi$). For those on the spine, the offspring law $(N,\{\bxi_n\})$ satisfies
\begin{align}
    \frac{\d\q}{\d\p}(N,\{\bxi_n\})=\sum_{n=1}^Ne^{\bth\cdot \bxi_n-\psi(\bth)}.\label{eq:dq/dp on spine}
\end{align}
 Observe that the first marginal satisfies ${\d\q}/{\d\p}(N)\propto N$, which implies (announced in Section \ref{Sub:spine_MLP}) that the offspring distribution is size-biased.  The process resulting from the previous recursive sampling procedure has the same law as $\q$. As before, we use $\E$ to denote the expectation under $\p$ and $\E^\q$ to denote the expectation under $\q$. 


Following \citep[Section 2.2]{luo2025precise}, we construct an auxiliary point process for the BRW with spine seen from the origin, using the following elements:
\begin{itemize}
    \item $\{(\{\bb_k(i)\}_{i\in[N_k]},w_k)\}_{k\geq 1}$ are i.i.d.~copies of a certain point process describing the offsprings on the spine (there are $N_k$ many at generation $k$) along with the index $w_k\in[N_k]$ that by definition governs which index continues the random walk on the spine. For convenience, we call them adjacent jumps or bones;
    \item  $\bS_k$ is the random walk formed by the chosen indices, i.e., $\bS_k=\bb_1(w_1)+\dots+\bb_k(w_k)$. By construction, $\bS_k$ is a non-lattice random walk with mean $(\hat{c}_1,\z)$ (see Lemma \ref{lemma:S_k} below);
    \item $\{\bet^{(i,k)}\}_{k\geq 1,\,i\in[N_k]}$ are i.i.d.~copies of the BRW with jump $\bxi$, which we call the off-spine BRWs (also called decorations in some literature);
    \item for each $k\in [n-1]$, $i\in[N_k]$, and $j\in[n-k-1]$, $V^{(i,k)}_j$ is the set of all particles (indexed by the age $j$) of the i.i.d.~copies of the BRW initiated at the location $\bb_{k+1}(i)+\bS_k$ at time $k+1$.  For a particle $v\in V^{(i,k)}_j,\,j\in[n-k-1]$, we use $\bet^{(i,k)}_v$ to denote its relative location in the off-spine BRW that emanates from $\bb_{k+1}(i)+\bS_k$, i.e., $\bet^{(i,k)}_v=\bet_v-(\bb_{k+1}(i)+\bS_k)$.
\end{itemize}
 See Figure \ref{fig:8} for an illustration of the above definitions. We may finally define the point process as
$$D_n=\delta_{\z}+\sum_{k=0}^{n} D_{n,k},$$
where $D_{n,n}=\delta_{\bS_n}$ and for $0\leq k\leq n-1$, 
\begin{align}
    D_{n,k}=\sum_{i\neq w_{k+1}}\sum_{u\in V^{(i,k)}_{n-k-1}}\delta_{\bb_{k+1}(i)+\bet^{(i,k)}_u+\bS_k}.\label{eq:Dnk}
\end{align}
Note that different from \citep{luo2025precise}, the point process is centered at the origin instead of the endpoint $\bS_n$ of the spine.

\begin{figure}[!htbp]
    \centering
    \includegraphics[width=\linewidth,trim={2.5cm 2.4cm 2.5cm 1.5cm},clip]{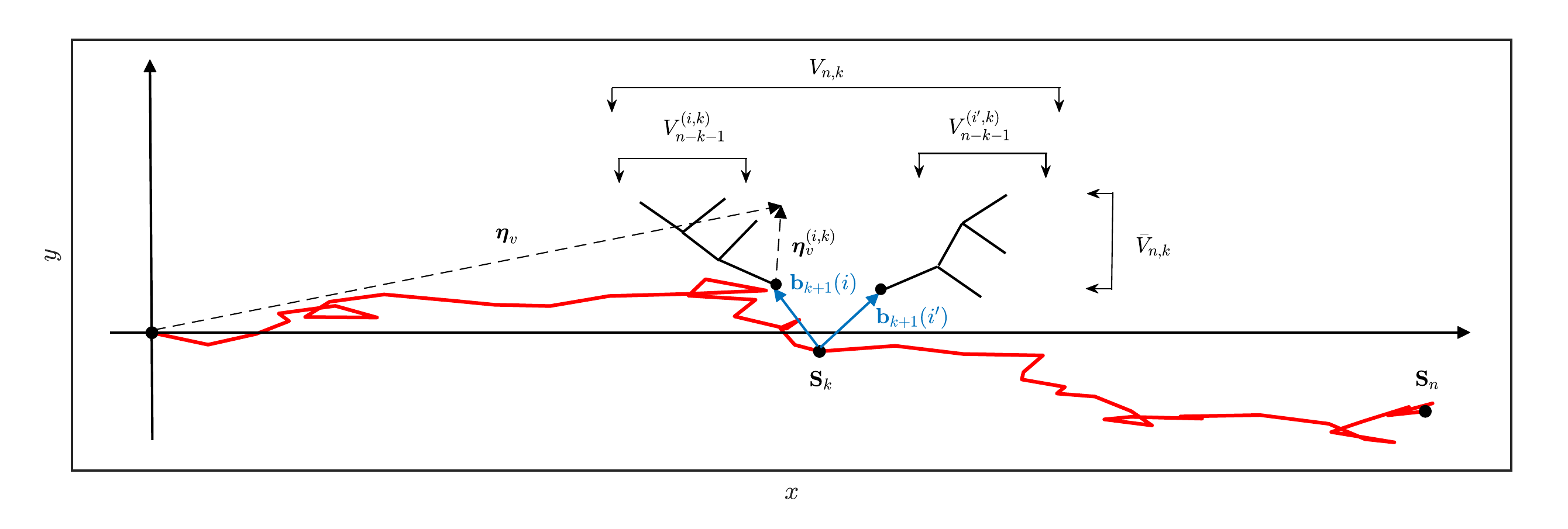}
    \caption{A pictorial demonstration of frequently used notations for the BRW with spine, with $d=2$. The red trajectory represents the spine random walk $\{\bS_k\}_{1\leq k\leq n}$. The blue arrows indicate the bones. The black trajectories represent off-spine BRWs, whose starting points are the tips of the bones.}
    \label{fig:8}
\end{figure}

The philosophy in describing (or simulating) the point process $D_n$ is as follows (analogous to Section \ref{Sub:spine_MLP}). First, a simple random walk $\{\bS_k\}_{1\leq k\leq n}$ is formed as the spine, which is the generator of fast particles. Second, consider the adjacent jumps $\{\bb_k(i)\}$. These jumps are not independent of the random walk $\{\bS_k\}$, but only constitute a one-step jump for each trajectory, which means that its contribution is negligible in the long run, a fact that can be easily justified using Borel--Cantelli arguments. Third, the i.i.d.~off-spine BRWs $\{\bet^{(i,k)}\}_{i,k\geq 1}$ are independent of anything else.

Since a first passage event of the form $\{\tau_x=n\}$ requires that all past trajectories (strictly) before time $n$ do not enter $B_x$, we need a variant of the point process \eqref{eq:Dnk} that further incorporates all past locations of a trajectory. 
For $0\leq k\leq n-1$, let $V_{n,k}=\bigcup_{i\neq w_{k+1}}V^{(i,k)}_{n-k-1}$ denote the set of all particles at time $n$ incorporated by $D_{n,k}$, i.e., those at time $n$ who left the spine at time $k$.
Let also $\bar{V}_{n,k}=\bigcup_{i\neq w_{k+1}}\bigcup_{j=0}^{n-k-1}V^{(i,k)}_{j}$ be the collection of all particles in $V_{n,k}$ together with their ancestors from the past times $k+1,\dots,n$. 
Define for $0\leq k\leq n-1$, 
\begin{align}
    \bar{D}_{n,k}=\sum_{i\neq w_{k+1}}\sum_{u\in \bar{V}_{n,k}}\delta_{\bb_{k+1}(i)+\bet^{(i,k)}_u+\bS_k}\label{eq:bar D}
\end{align}
and $\bar{D}_{n,n}=D_{n,n}=\delta_{\bS_n}$.

The following lemma describes the law of the spine random walk $\{\bS_k\}_{1\leq k\leq n}$ under $\q$. The result is well-known, but for completeness we provide a proof in general dimensions in Appendix \ref{app:proofs}.

\begin{lemma}\label{lemma:S_k}
  Let $\q_{\mathrm{spine}}$ denote the increment law of the random walk $\{\bS_k\}_{k\in\N}$ under $\q$ with $\bth=\hat{\bc}_2$, and $\p_\bxi$ denote the law of $\bxi$ under $\p$. The following holds.
  \begin{enumerate}[(i)]
      \item ${\d\q_{\mathrm{spine}}}/{\d\p_\bxi}(\bx)\propto e^{\hat{\bc}_2\cdot\bx}$; in particular, the random walk $\{\bS_k\}_{k\in\N}$ is non-lattice.
      \item  $\q_{\mathrm{spine}}$ has mean $(\hat{c}_1,\z)$ and finite variance.
      \item For $n=x/\hat{c}_1+O(\sqrt{x})$, 
    $$\q(\bS_n\in B_\bx)\asymp x^{-d/2}\quad\text{ as }x\to\infty.$$
  \end{enumerate}
\end{lemma}




\subsubsection{Lower deviation}
\label{sec:lowerd}

The goal of this section is to prove Theorem \ref{thm:Ld1}. We first prove the following slightly stronger lower bound, which will also be useful later in Theorem \ref{prop:alg}. Let us use the short-hand notation $\hat{\bc}_1=(\hat{c}_1,\z)\in\R^d$.

\begin{proposition}\label{prop:LB}
    For each fixed integer $t$, we have
    \begin{align}
        \p\Big(\tau_x=\lfloor\frac{x}{\hat{c}_1}\rfloor-t\Big)\gg x^{-d/2}e^{-\frac{x}{\hat{c}_1}(I(\hat{\bc}_1)-\log\rho)}\quad\text{ as }x\to\infty.\label{eq:LB}
    \end{align}
\end{proposition}

The main difficulty here is that, in addition to placing one particle in the ball $B_x$ at time $\lfloor x/\hat{c}_1\rfloor -t$, one also needs to ensure that none of the previous particles ever entered the ball $B_x$ strictly before time $\lfloor x/\hat{c}_1\rfloor-t$. As a result, we need to carefully construct an event on which $\tau_x=\lfloor x/{\hat{c}_1}\rfloor-t$. We need a few preparations before going to the proof. For notational simplicity, we write $n=\lfloor x/{\hat{c}_1}\rfloor-t$.

The intuition is to construct the event $\{\tau_x=n\}$ with $n=x/\hat{c}_1+O(1)$ by isolating a path with the following properties under $\q$.
\begin{enumerate}[(i)]
    \item It holds that $\bS_n\in B_x$.
    \item The increments $\bb_{n-\ell}(i),~i\neq w_{n-\ell},~1\leq \ell\leq L$ are constrained; we want to ensure that the BRW does not visit the target $B_x$ at least at times $n-1,n-2,\dots,n-L$.
    \item The first increments of the spine grow at least at speed $\bar{c}_1\in(c_1,\hat{c}_1)$ (i.e., $S_{n-1}-S_{n-\ell}\geq \bar{c}_1(\ell-1)$ for all $\ell\in[L,m_n]$). We will discuss the choice of $m_n$ momentarily. But in principle, even $m_n=n$ is feasible since $\bar{c}_1<\hat{c}_1$.
    \item All bones are constrained. Namely, we impose constraints on $\bb_{n-\ell}(i)$ for $i\neq w_{n-\ell}$ for $L\leq\ell\leq m_n$. Again, the constraint will be chosen so that this happens with probability bounded away from zero even if $m_n=n$.
\end{enumerate}
Having (i)--(iv) in place, the off-spine BRWs emanating from the bones are standard BRWs, which will grow at speed $c_1$ (recall that $c_1<\bar{c}_1<\hat{c}_1$). We simply need to apply a local CLT (Lemma \ref{lemma:S_k}(iii)) to verify $\bS_n\in B_x$. This is where choosing $m_n=n^{1/3}$ is helpful. On the one hand, $n^{1/3}$ is large enough so that, given the constraints in the last $m_n$ steps, $B_x$ is very likely not reached strictly before time $n$. On the other hand, $n^{1/3}$ is small enough so that $\bS_{n-n^{1/3}}+(\bS_n-\bS_{n-n^{1/3}})\in B_x$ still is subject to a local CLT provided that
\begin{enumerate}
    \item [(v)] $\n{\bS_n-\bS_{n-n^{1/3}}}\leq Ln^{1/3}$,
\end{enumerate}
which has a large probability under $\q$ for $L$ large (it is important that $n^{1/3}=o(n^{1/2})$ to apply the local CLT).

With the elements (i)--(v) in place, the proof of Proposition \ref{prop:LB} proceeds by bounding 
$$\frac{\d\p}{\d\q}\gg \exp(-n(I(\hat{c}_1)-\log\rho))$$
on a suitable event informed by (i)--(iv), then taking expectations under $\q$ incorporating (v) will lead to the desired lower bound \eqref{eq:LB}.

In preparation for executing this approach, we need basic bounds related to the likelihood ratio under $\q$. This is given by the next lemma, whose proof is delayed until Appendix \ref{app:proofs}.

\begin{lemma}\label{lemma:b_k}
Let $\ee>0$ be arbitrary. Then there exists $\delta_1>0$ such that uniformly for $\ell_0\geq 0$,
 $$\sum_{\ell\geq \ell_0}\q\bigg(\sum_{i=1}^{N_\ell} e^{\hat{\bc}_2 \cdot \bb_{\ell}(i)}\geq e^{\ee\ell }\bigg)\ll e^{-\delta_1\ell_0}.$$
\end{lemma}

We are now ready to define the events that will explicitly inform constraints (i)--(v) above. The following Figure~\ref{fig:enter-label} is provided to guide the reader, and the definitions of the events (of the form $E_i,E_{i,j}$) are given below.

\begin{figure}[!htbp]
    \centering
\includegraphics[width=0.27\linewidth, angle=90,trim={7.5cm 1.5cm 6.8cm 1.5cm},clip]{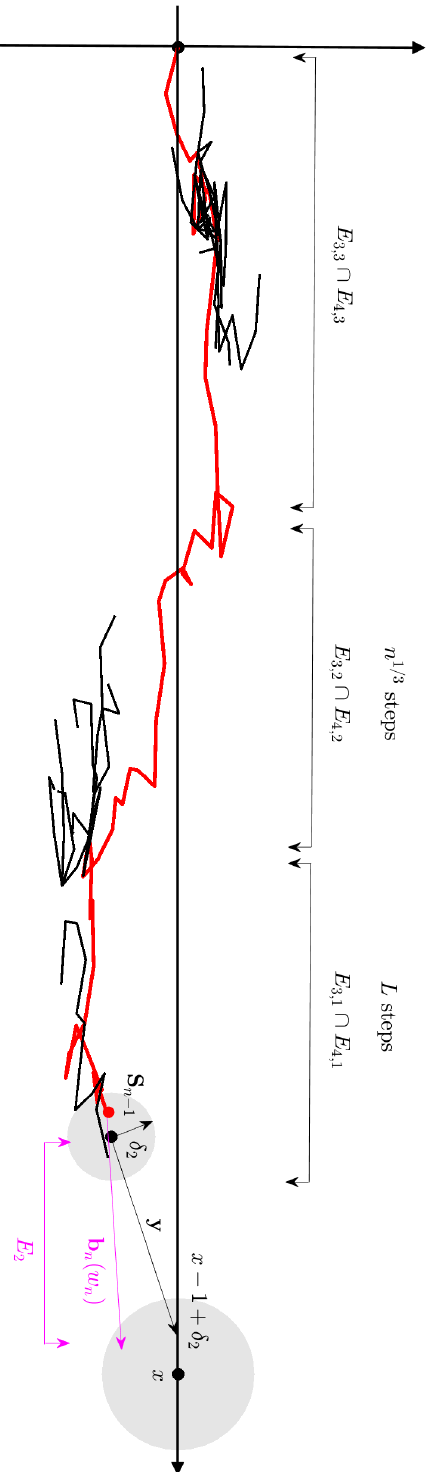}
    \caption{Schematic of the proof of Proposition \ref{prop:LB} in dimension $d=2$. The events $E_3,E_4$ constrain the spine and the adjacent jumps. For $i=3,4$, $E_{i,1}$ (resp.~$E_{i,2},\,E_{i,3}$) controls the behavior of the first $L$ steps (resp.~next $n^{1/3}$ steps, the remaining steps). The event $E_2$ ensures the first passage event through controlling the final jump at time $n$.}
    \label{fig:enter-label}
\end{figure}

Denote by $\supp\bxi$ the support of the random variable $\bxi$. 
Let $\by\in\supp\bxi\cap ((0,\infty)\times\R^{d-1})$ and pick $\delta_2>0$ small enough so that $\q(\bxi\in B_{\delta_2}(\by))>0$ and $ \by\cdot\be>2\delta_2$. 
Write
\begin{align}
    S_k=\frac{\hat{\bc}_2\cdot\bS_k}{\bce}.\label{eq:Sk}
\end{align}
Recall the definition of $\bar{c}_1$ from \eqref{eq:ccc}. Let $b_k$ be the first coordinate of $\bb_k$.  Define the events (with constants $L,L_1,\delta_3$ to be determined)
\begin{align*}
    E_1&=\{\bS_{n-1}\in B_{\delta_2}(x-1+\delta_2,\z)-\by\};\\
    E_2&=\{\bb_n(w_n)\in B_{\delta_2}(\by)\text{ and }\forall i\neq w_n,\,b_n(i)<0\};\\
    E_{3,1}&=\{\forall 1\leq \ell\leq L, S_{n-\ell}-S_{n-1-\ell}\geq \hat{c}_1\};\\
    E_{3,2}&=\{\forall L\leq\ell\leq n^{1/3}, S_{n-1}-S_{n-1-\ell}> \bar{c}_1\ell;S_{n-1}-S_{n-1-n^{1/3}}>\delta_3(\bar{c}_1+\delta_3)n^{1/3}\};\\
    E_{3,3}&=\{\forall n^{1/3}\leq\ell\leq n-1, S_{n-1}-S_{n-1-\ell}> \bar{c}_1\ell\};\\
    E_{3,4}&=\left\{\n{\bS_{n-1}-\bS_{n-1-n^{1/3}}}\leq L_1 n^{1/3}\right\};\\
     E_3&= \bigcap_{j=1}^4E_{3,j};\\
    E_{4,1}&=\{\forall 1\leq \ell\leq L, \forall i\neq w_{n-\ell},b_{n-\ell}(i)<\hat{c}_1\};\\
    E_{4,2}&=\Big\{\forall L\leq\ell\leq n^{1/3}, \sum_ie^{\hat{\bc}_2\cdot \bb_{n-\ell}(i)}<e^{\ee_1\ell\bce }\Big\};\\
    E_{4,3}&=\Big\{\forall n^{1/3}\leq\ell\leq n-1, \sum_ie^{\hat{\bc}_2\cdot \bb_{n-\ell}(i)}<e^{\ee_1\ell\bce }\Big\};\\
    E_4&= \bigcap_{j=1}^3E_{4,j};\\
    E_5&=\bigcap_{i=1}^4 E_i;\\
    E_{6,1}&=\Big\{\forall 1\leq \ell\leq L,\forall i,\max_{u\in\bar{V}_{n,\ell}} \eta^{(i,n-1-\ell)}_u\leq \ell\hat{c}_1\Big\};\\
    E_{6,2}&=\Big\{\sum_{k=0}^{n-L}\sum_{v\in \bar{V}_{n,k}}e^{\hat{\bc}_2\cdot\bet_v}\leq \frac{1}{2}e^{\bce x}\Big\};\\
    E_6&=E_{6,1}\cap E_{6,2}.
\end{align*}

The intuition is that we try to construct samples of the event $\{\tau_x=n\}$ (with $n=x/\hat{c}_1+O(1)$) by first constructing a well-behaved trajectory of the spine $\{\bS_k\}_{1\leq k\leq n}$ and the adjacent jumps $\{\bb_k(i)\}$. Events $E_1,E_2$ ensure that the spine reaches the target $B_x$ at time $n$, while events $E_3,E_4$ control the trajectory of the spine (along with the adjacent bones) so that the off-spine BRWs attached to the spine are unlikely to reach $B_x$ strictly before time $n$ (recall that on the event $\{\tau_x=n\}$, no particle reaches $B_x$ strictly before time $n$). Events $E_{4,2}$ and $E_{4,3}$ also help to control $\d\p/\d\q$ from below. The control of the off-spine BRWs, which is independent of everything else, is given by the event $E_6$. See also Figure \ref{fig:enter-label} for an illustration.

Let us specify the choices of $L$, $L_1$, and $\delta_3$ here. Since $\E^\q[\xi]<\bar{c}_1$, there exists $\theta<0$ such that $\E^\q[e^{\theta (\xi-\bar{c}_1)}]=1$. 
By the optional stopping theorem applied to the martingale $Y_n=\prod_{k=1}^n e^{\theta (\xi_k-\bar{c}_1)}$, we have 
$\q(E_{3,2}\mid E_{3,1})\geq 1-e^{\theta L(\hat{c}_1-\bar{c}_1)}$, if $\delta_3\in(0,\hat{c}_1-\bar{c}_1)$ is picked small enough. We also pick $L_1>0$ large enough such that
\begin{align}
    \q(E_{3,4}^{\mathrm{c}})=o(1),\label{eq:E34}
\end{align}
which is possible by Cram\'{e}r's large deviation upper bound (Theorem 2.2.3 of \citep{dembo2009large}).

On the other hand, Lemma \ref{lemma:b_k} implies $\q(E_{4,2}^{\mathrm{c}})\to 0$ as $L\to\infty$. Therefore, for $L>0$ chosen large enough, 
\begin{align}
    \q(E_{3,2}\cap E_{4,2}\mid E_{3,1})\gg 1.\label{eq:q32}
\end{align} Next, we set $L$ large enough such that on the event $E_3\cap E_4$, the off-spine BRWs that left the spine for at least $L$ steps rarely reach as far as $x$.  Recall the definition of $\bar{V}_{n,k}$ from above \eqref{eq:bar D} and that the off-spine BRWs are independent from anything else. We have by \eqref{eq:psibth} and \eqref{eq:ccc}, 
\begin{align*}
    &\E^\q\bigg[\sum_{k=0}^{n-L}\sum_{v\in \bar{V}_{n,k}}e^{\hat{\bc}_2\cdot\bet_v}\mid E_5,\{\bS_k\},\{\bb_k(i)\}\bigg]\\
    &=\E^\q\bigg[\sum_{k=0}^{n-L} \int e^{\bce x}\d \bar{D}_{n,k}(x)\mid E_5,\{\bS_k\},\{\bb_k(i)\}\bigg]\\
    &=\bone_{E_5}\sum_{k=0}^{n-L} e^{\hat{\bc}_2\cdot \bS_k}\bigg(\sum_{i=1}^{N_k} e^{\hat{\bc}_2\cdot \bb_{k+1}(i)}\bigg)\sum_{j=0}^{n-k-1}\E^\q\bigg[\sum_{v\in V^{(i,k)}_{j}}e^{\hat{\bc}_2\cdot\bet^{(i,k)}_v}\bigg]\\
    &\ll \sum_{k=0}^{n-L} e^{\bce(x-(n-k)\bar{c}_1)}e^{\ee_1(n-k)\bce} e^{(n-k-1)(\hat{c}_1\bce+\log\rho-I(\hat{c}_1))}\\
    &\ll  e^{\bce x-\delta_4 L},
\end{align*}for some $\delta_4>0$. The same conditioned exponential moment computation will be used several times below. By Markov's inequality and choosing $L$ large enough, we have
\begin{align}
    \q(E_{6,2}^{\mathrm{c}}\mid E_5,\{\bS_k\},\{\bb_k(i)\})<\frac{1}{4}.\label{eq:E62}
\end{align}
\begin{lemma}\label{lemma:E1E3}
    It holds that $\q(E_5)\gg x^{-d/2}$. 
\end{lemma}
\begin{proof}
    We reverse the random walk in time and define $T_\ell=S_{n-1}-S_{n-1-\ell}-\bar{c}_1\ell$. So $\{T_\ell\}_{1\leq\ell\leq n-1}$ is a random walk with positive drift $\hat{c}_1-\bar{c}_1$. The idea is to recursively condition on the random walk at times $\ell=L,n^{1/3}$ while applying the constraints to the events simultaneously.
    
 First, with $L$ fixed, it follows that $\q(E_{3,1}\cap E_{4,1})\gg 1$ since $\E^\q[\bxi]=\hat{\bc}_1$. 
 It follows from \eqref{eq:q32} that
 $$\q(E_{3,2}\cap E_{4,2}\mid E_{3,1}\cap E_{4,1})=\q(E_{3,2}\cap E_{4,2}\mid E_{3,1})\gg 1,$$
 which along with \eqref{eq:E34} further implies that
 $$\q(E_{3,1}\cap E_{3,2}\cap E_{3,4}\cap E_{4,1}\cap E_{4,2})\gg 1.$$
    Now, conditioned on the trajectory of $\{T_\ell\}_{1\leq \ell\leq n^{1/3}}$ and the event $E_{3,1}\cap E_{3,2}\cap E_{3,4}\cap E_{4,1}\cap E_{4,2}$, the conditional probability of $E_1$ is uniformly $\gg x^{-d/2}$ by Lemma \ref{lemma:S_k}(iii). 
    We arrive at
    $$\q(E_{3,1}\cap E_{3,2}\cap E_{4,1}\cap E_{4,2}\cap E_1)\gg x^{-d/2}.$$

Next, we eliminate events $E_{3,3}^{\mathrm{c}} $ and $ E_{4,3}^{\mathrm{c}}$. Write $\xi=(\hat{\bc}_2\cdot\bxi)/(\bce)$. 
By the optional stopping theorem applied to the martingale $Y_n:=\prod_{k=1}^n e^{\theta (\xi_k-\bar{c}_1)}$ where $\theta<0$ is such that $\E^\q[e^{\theta (\xi-\bar{c}_1)}]=1$, we see that $\q(E_{3,2}\cap E_{3,3}^{\mathrm{c}})\leq e^{\theta \delta_3(\bar{c}_1+\delta_3)n^{1/3}}=o(x^{-d/2})$.
On the other hand, by Lemma \ref{lemma:b_k}, we have $\q(E_{4,3}^{\mathrm{c}})=O(e^{-\delta_1n^{1/3}})=o(x^{-d/2})$, and hence $$\q(E_1\cap E_3\cap E_4)\gg x^{-d/2}.$$ 

Finally, $E_2$ is independent of anything else (since it depends only on the last step of the spine) and $\p(E_2)\gg 1$, so we conclude that $\q(E_5)\gg x^{-d/2}$. 
\end{proof}

\begin{proof}[Proof of Proposition \ref{prop:LB}] We condition on the event $E_5$, where we have showed in Lemma \ref{lemma:E1E3} that $\q(E_5)\gg x^{-d/2}$. 
It is also clear from definition that $\tau_x=n$ on the event $E_5\cap E_6$, so it suffices to show that
\begin{align}
   \p(E_5\cap E_6)= \E^\q\bigg[\frac{\d\p}{\d\q}\bone_{E_5\cap E_6}\bigg]\gg x^{-d/2} e^{-n(I(\hat{\bc}_1)-\log\rho)}.\label{eq:E5E6}
\end{align}
To this end, we need to give a lower bound on $\d\p/\d\q$. Recall \eqref{eq:dq/dp}. We have, conditionally by \eqref{eq:psibth} and \eqref{eq:ccc},
\begin{align*}
    \E^\q\bigg[\sum_{v\in V_n}e^{\hat{\bc}_2\cdot\bet_v}\mid E_5,\{\bS_k\},\{\bb_k(i)\}\bigg]
    &=\E^\q\bigg[\sum_{k=0}^n \int e^{\bce x}\d D_{n,k}(x)\mid E_5,\{\bS_k\},\{\bb_k(i)\}\bigg]\\
    &=\bone_{E_5}\sum_{k=0}^n e^{\hat{\bc}_2\cdot \bS_k}\bigg(\sum_{i=1}^{N_k} e^{\hat{\bc}_2\cdot \bb_{k+1}(i)}\bigg)\E^\q\bigg[\sum_{v\in V^{(i,k)}_{n-k-1}}e^{\hat{\bc}_2\cdot\bet^{(i,k)}_v}\bigg]\\
    &\ll \sum_{k=0}^n e^{\bce(x-(n-k)\bar{c}_1)}e^{\ee_1(n-k)\bce} e^{(n-k-1)(\hat{c}_1\bce+\log\rho-I(\hat{c}_1))}\\
    &\ll e^{\hat{c}_2 x}.
\end{align*}
By Markov's inequality, there exists $L_2>0$ such that
$$\q\bigg(\sum_{v\in V_n}e^{\hat{\bc}_2\cdot\bet_v}\leq L_2e^{\bce x}\mid E_5,\{\bS_k\},\{\bb_k(i)\}\bigg)>\frac{1}{2}.$$
Combining with \eqref{eq:E62} yields
\begin{align}
   \q\bigg(E_{6,2}\cap\Big\{\sum_{v\in V_n}e^{\hat{\bc}_2\cdot\bet_v}\leq L_2e^{\bce x}\Big\}\mid E_5,\{\bS_k\},\{\bb_k(i)\}\bigg) >\frac{1}{4}.\label{eq:E6}
\end{align}
Note also that on the event $\{\sum_{v\in V_n}e^{\hat{\bc}_2\cdot\bet_v}\leq L_2e^{\bce x}\}$, we have by \eqref{eq:dq/dp} and \eqref{eq:psibth},
\begin{align}
    \frac{\d\p}{\d\q}=\frac{1}{\sum_{u\in V_n}e^{\bc_2\cdot\bet_u-n\psi(\bc_2)}}\gg \frac{1}{e^{\hat{\bc}_2\cdot\be x}e^{-n\psi(\bc_2)}}=e^{-n(I(\hat{\bc}_1)-\log\rho)}.\label{eq:u1}
\end{align}
We thus obtain using independence, \eqref{eq:E6}, and \eqref{eq:u1} that
\begin{align*}
    &\E^\q\bigg[\frac{\d\p}{\d\q}\bone_{E_5\cap E_6}\bigg]\\
    &=\E^\q\Bigg[\E^\q\bigg[\frac{\d\p}{\d\q}\bone_{ E_{6,2}}\mid E_5,\{\bS_k\},\{\bb_k(i)\}\bigg]\bone_{E_5\cap E_{6,1}}\Bigg]\\
    &\geq \E^\q\Bigg[\E^\q\bigg[\frac{\d\p}{\d\q}\bone_{ E_{6,2}}\bone_{\{\sum_{v\in V_n}e^{\hat{\bc}_2\cdot\bet_v}\leq L_2e^{\bce x}\}}\mid E_5,\{\bS_k\},\{\bb_k(i)\}\bigg]\bone_{E_5\cap E_{6,1}}\Bigg]\\
    &\gg e^{-n(I(\hat{\bc}_1)-\log\rho)}\E^\q\Bigg[\q\bigg(E_{6,2}\cap\Big\{\sum_{v\in V_n}e^{\hat{\bc}_2\cdot\bet_v}\leq L_2e^{\bce x}\Big\}\mid E_5,\{\bS_k\},\{\bb_k(i)\}\bigg)\bone_{E_5\cap E_{6,1}}\Bigg]\\
    &\gg e^{-n(I(\hat{\bc}_1)-\log\rho)}\q(E_5\cap E_{6,1}).
\end{align*}
Observe that $E_5$ and $E_{6,1}$ are independent and $\q(E_{6,1})\gg 1$, since $E_{6,1}$ only constrains finitely many independent BRWs. Therefore, we conclude using Lemma \ref{lemma:E1E3} that
$$\p(\tau_x=n)\geq  \E^\q\bigg[\frac{\d\p}{\d\q}\bone_{E_5\cap E_6}\bigg]\gg e^{-n(I(\hat{\bc}_1)-\log\rho)}\q(E_5)\q(E_{6,1})\gg x^{-d/2}e^{-n(I(\hat{\bc}_1)-\log\rho)},$$
as desired.
\end{proof}

 \begin{proof}[Proof of Theorem \ref{thm:Ld1}]
 The lower bound of Theorem \ref{thm:Ld1} follows immediately from Proposition \ref{prop:LB} applied with $t=1$, so it remains to establish the upper bound.
First, applying the union bound, we have 
\begin{align*}
    \p\Big(\tau_x<\frac{x}{\hat{c}_1}\Big)
    &\leq \sum_{j=1}^{\lfloor x/\hat{c}_1\rfloor}\p(\exists v\in V_j,\,\bet_{v}\in B_x)\\
    &\leq \sum_{j=1}^{\lfloor x/\hat{c}_1-\sqrt{x}\rfloor}\rho^j\p(\bS_j\in B_x)+\sum_{j=\lfloor x/\hat{c}_1-\sqrt{x}\rfloor+1}^{\lfloor x/\hat{c}_1\rfloor}\rho^j\p(\bS_j\in B_x).
\end{align*}
For the first sum over $1\leq j\leq \lfloor x/\hat{c}_1-\sqrt{x}\rfloor$, we apply Cram\'{e}r's large deviation upper bound, which implies that
\begin{align*}
    \p(\bS_j\in B_x)&\ll \exp\Big(-(\frac{x}{\hat{c}_1}-\sqrt{x})I(\frac{\bx}{x/\hat{c}_1-\sqrt{x}})\Big)\\
    &\leq \exp\Big(-(\frac{x}{\hat{c}_1}-\sqrt{x})\big(I(\hat{\bc}_1)+\frac{\hat{c}_1\bce}{\sqrt{x}/\hat{c}_1-1}\big)\Big)\leq \exp\Big(-\frac{x}{\hat{c}_1}I(\hat{\bc}_1)-\delta_5\sqrt{x}\Big)
\end{align*}for some $\delta_5>0$, 
where we have used the convexity of $I$ and that $\nabla I(\hat{\bc}_1)=\hat{\bc}_2$. For the second sum over $\lfloor x/\hat{c}_1-\sqrt{x}\rfloor<j\leq \lfloor x/\hat{c}_1\rfloor$, we apply the Bahadur--Rao theorem (Theorem 3.7.4 of \citep{dembo2009large} and Remark (c) that follows) which implies that
$$\p(\bS_j\in B_x)\ll x^{-d/2}\exp(-jI(\hat{\bc}_1)- \hat{\bc}_2\cdot(\bx-j\hat{\bc}_1)), \quad\lfloor x/\hat{c}_1-\sqrt{x}\rfloor\leq j\leq \lfloor x/\hat{c}_1\rfloor.$$
Altogether, we obtain
\begin{align*}
    \p\Big(\tau_x<\frac{x}{\hat{c}_1}\Big)&\ll x\rho^{x/\hat{c}_1}e^{-\frac{x}{\hat{c}_1}I(\hat{\bc}_1)-\delta_5\sqrt{x}}+\sum_{j=\lfloor x/\hat{c}_1-\sqrt{x}\rfloor+1}^{\lfloor x/\hat{c}_1\rfloor}\rho^jx^{-d/2}e^{-jI(\hat{\bc}_1)- \hat{\bc}_2\cdot(\bx-j\hat{\bc}_1)}\\
    &\ll x^{-d/2}e^{-\frac{x}{\hat{c}_1}(I(\hat{\bc}_1)-\log\rho)},
\end{align*}
as desired.
 \end{proof}

\subsubsection{Upper deviation}
As explained in the introduction, the intuition of the upper large deviation of FPT is similar to \citep{gantert2018large}: in order to realize the event $\{\tau_x>x/\hat{c}_1\}$, we assert that for time up to $\alpha x$ for some $\alpha\in(0,1/\hat{c}_1)$, branching events are restricted and there is a single particle at time $\alpha x$. The location of such a particle at time $\alpha x$ is near $x\by_x$ where $\by_x\in \R^d$ is fixed and $x\to\infty$. After time $\alpha x$, the BRW behaves normally. These events can be formulated as follows.
\begin{align}
    \begin{split}
        \tilde{E}_1&=\{|V_{\lfloor\alpha x\rfloor}|=1\};\\
    \tilde{E}_2&=\{\forall v\in V_{\lfloor\alpha x\rfloor}, \bet_v\in B_{x\by_x}\};\\
    \tilde{E}_3&=\{\tau_x>x/\hat{c}_1\}.
    \end{split}\label{eq:E1E2E3}
\end{align}
To deal with the event $\tilde{E}_1$, we need the following preliminary result on the lower deviation rates of the Galton--Watson process, which are consequences of \citep{athreya2004branching,bansaye2013lower}, as noted in \citep{gantert2018large}. We say that a sequence $\{a_n\}$ is subexponential if $a_n=o(e^{\ee n})$ for any $\ee>0$.

\begin{lemma}\label{lemma:GW LD}
    Assume (A1) and (A5). Then 
    $$\lim_{n\to\infty}\frac{1}{n}\log\p(|V_n|=1\mid S)=-\gamma.$$
    Moreover, for every subexponential sequence $\{a_n\}_{n\in\N}$ with $a_n\to\infty$, we have
    $$\lim_{n\to\infty}\frac{1}{n}\log\p(|V_n|\leq a_n\mid S)=-\gamma.$$
\end{lemma}

\begin{proof}
    It is clear that, under Assumption (A5), $\p(|V_2|=1)>0$. Therefore, the result follows immediately from Theorem 2.5 of \citep{gantert2018large}.
\end{proof}

We also need the following uniform version of upper bounds on FPT. Recall that $q=1-\p(S)$ is the probability of extinction.

\begin{lemma}\label{lemma:uni}
    For $\beta>0$, let
    $U_\beta=\{\bz\in\R^d:I(\bz)\leq \beta\}.$
    Then for any $\ee>0$,
    $$\sup_{\bz\in U_{\log\rho-\ee}}\p(\tau_{x\bz}>x)\to q\quad\text{as }x\to\infty.$$
    \end{lemma}

    \begin{proof}
The bootstrapping argument resembles the proof of Theorem 1 of \citep{blanchet2024first} and we only sketch the proof. It remains to show
$$\sup_{\bz\in U_{\log\rho-\ee}}\p(\tau_{x\bz}>x\mid S)\to 0\quad\text{as }x\to\infty.$$
First, one run the BRW for a time $\asymp \log x$ to identify $x^{(d-1)/2}$ particles that are $O(\log x)$ close to the origin (independent of the choice of $\bz$). These particles then independently initiate BRW processes for a time of $(1-\delta_6)x$. With some $\delta_6$ small enough such that $I(\bz/(1-\delta_6))<\log\rho$, the probability that such a BRW process, projected onto the direction $\tilde{\bc}_2=\nabla I(\bz/(1-\delta_6))$, reaches $\approx x\bz\cdot\tilde{\bc}_2/\n{\tilde{\bc}_2}$ with a probability close to one. Extract one such trajectory from each BRW that verifies this property. By a local CLT (e.g., Lemma 30 of \citep{blanchet2024tightness}), the probability that such a trajectory reaches $B_x$ is $\gg x^{-(d-1)/2}$. It remains to check that the implicit constant here can be made uniform in the choice of $\bz$ (or equivalently the choice of $\tilde{\bc}_2$). Nevertheless, the constant arising from the local CLT is $\gg \var(\bxi\cdot\tilde{\bc}_2/\n{\tilde{\bc}_2})/\det(\Sigma_\bxi)\gg \var(\bxi\cdot\tilde{\bc}_2/\n{\tilde{\bc}_2})$, which is uniformly bounded from below since the jump $\bxi$ is not supported on a hyperplane (which is a consequence of the non-lattice condition (A4)).     
    \end{proof}

  \begin{proof}[Proof of Theorem \ref{thm:LD2}]
We first prove the lower bound of $\tau_x$.
Recall \eqref{eq:E1E2E3} and consider the event $\tilde{E}=\tilde{E}_1\cap \tilde{E}_2\cap \tilde{E}_3$. It follows from independence that
\begin{align*}
    \p(\tilde{E}\mid S)&=\p(\tilde{E}_1\mid S)\p(\tilde{E}_2\mid \tilde{E}_1,S)\p(\tilde{E}_3\mid \tilde{E}_1,\tilde{E}_2,S)\\
    &\leq \p(\tilde{E}_1\mid S)\p(\bS_{\lfloor\alpha x\rfloor}\in B_{x\by_x})\p\Big(\tau_{\bx-x\by_x}^{(2)}>(\frac{1}{\hat{c}_1}-\alpha)x\Big),
\end{align*}
where $\tau_{\bx-x\by_x}^{(2)}$ refers to the FPT to the ball $B_{2}(\bx-x\by_x)$. Applying Lemma \ref{lemma:GW LD} and the large deviation rate for the random walk $\{\bS_k\}_{k\geq 0}$ yields 
\begin{align}
    \lim_{x\to\infty}\frac{1}{x}\log\p(\tilde{E}\mid S)\geq -\gamma\alpha-\alpha I\big(\frac{\by_x}{\alpha}\big)+ \lim_{x\to\infty}\frac{1}{x}\log\p\Big(\tau_{\bx-x\by_x}^{(2)}>(\frac{1}{\hat{c}_1}-\alpha)x\Big).\label{eq:ld1}
\end{align}
For $\alpha,\by_x$ satisfying the constraint in \eqref{eq:max}, i.e., $I((\be-\by_x)/(1/\hat{c}_1-\alpha))\geq\log\rho$, we have by a union bound argument over the set of particles at time $(1/\hat{c}_1-\alpha)x$ along with Cram\'{e}r's upper bound that
$$\lim_{x\to\infty}\frac{1}{x}\log\p\Big(\tau_{\bx-x\by_x}^{(2)}>(\frac{1}{\hat{c}_1}-\alpha)x\Big)\geq 0.$$
Therefore, by \eqref{eq:ld1},
$$\lim_{x\to\infty}\frac{1}{x}\log\p\Big(\tau_x>\frac{x}{\hat{c}_1}\mid S\Big)\geq\lim_{x\to\infty}\frac{1}{x}\log\p(\tilde{E}\mid S)\geq -\Big(\gamma\alpha+\alpha I\big(\frac{\by_x}{\alpha}\big)\Big).$$
Taking a supremum over $\alpha,\by_x$ on the right-hand side thus yields
$$\lim_{x\to\infty}\frac{1}{x}\log\p\Big(\tau_x>\frac{x}{\hat{c}_1}\mid S\Big)\geq -T(x,\hat{c}_1;\rho,\gamma).$$

For the upper bound on $\tau_x$, we follow the same approach as \citep{gantert2018large}. Define
$$T_x=\inf\{t\geq 0:|V_{\lfloor tx\rfloor}|\geq x^{3}\}.$$
Let $\ee_2>0$ and introduce the set
$$F=F(\ee_2)=\{k\ee_2\}_{1\leq k\leq \lceil (\hat{c}_1\ee_2)^{-1}\rceil}.$$
We have
\begin{align}
\begin{split}
    &\p\Big(\tau_x>\frac{x}{\hat{c}_1}\mid S\Big)\\
    &\leq \p\Big(T_x>\frac{1}{\hat{c}_1}\mid S\Big)+\sum_{t\in F}\p\Big(\tau_x>\frac{x}{\hat{c}_1}\mid T_x\in(t-\ee_2,t],S\Big)\p(T_x\in(t-\ee_2,t]\mid S)\\
    &\leq \p(|V_{\lfloor x/\hat{c}_1\rfloor}|< x^3\mid S)+\sum_{t\in F}\p\Big(\tau_x>\frac{x}{\hat{c}_1}\mid T_x\in(t-\ee_2,t],S\Big)\p(|V_{\lfloor(t-\ee_2)x\rfloor}|\leq x^3\mid S).
\end{split}\label{eq:ppp}
\end{align}
By (5.11) of \citep{gantert2018large},
\begin{align}
    \p(|V_{\lfloor tx\rfloor}|<x^2\mid T_x\in(t-\ee_2,t])\leq \exp\Big(\frac{x^3}{2}\log q\Big).\label{eq:tx<x^2}
\end{align}

For $\ee_3>0$, define 
the set
\begin{align}
    R_t=R_t(\rho,\ee_3)=\be-\Big(\frac{1}{\hat{c}_1}-t\Big)I^{-1}((\log\rho-\ee_3,\infty))\subseteq\R^d.\label{eq:Rt}
\end{align}
By \eqref{eq:tx<x^2}, we have
\begin{align}
    \begin{split}
        &\p\Big(\tau_x>\frac{x}{\hat{c}_1}\mid T_x\in(t-\ee_2,t],S\Big)\\
    &\leq \p(\exists v\in V_{\lfloor tx\rfloor}:\bet_{v}\in xR_t\mid T_x\in(t-\ee_2,t],S)\\
    &\hspace{1cm}+\p\Big(\forall v\in V_{\lfloor tx\rfloor},\,\bet_{v}\not\in xR_t; \,\tau_x>\frac{x}{\hat{c}_1};\,|V_{\lfloor tx\rfloor}|\geq x^2\mid T_x\in(t-\ee_2,t],S\Big)\\
    &\hspace{1cm}+\p(|V_{\lfloor tx\rfloor}|<x^2\mid T_x\in(t-\ee_2,t],S)\\
    &\leq e^{2(\log\rho)\ee_2x}\p(\bS_{\lfloor tx\rfloor}\in xR_t)+\sup_{{\by}_x\not\in R_t}\p\Big(\tau_{\bx-{\by_x}x}>\frac{x}{\hat{c}_1}-tx\Big)^{x^2}+p^{-1}\exp\Big(\frac{x^3}{2}\log q\Big).
    \end{split}\label{eq:p2}
\end{align}
For ${\by_x}\not\in R_t$, by definition,
$$
I\Big(\frac{\be-\by_x}{1/\hat{c}_1-t}\Big)\leq\log\rho-\ee_3,$$
so that by Lemma \ref{lemma:uni},
$$\sup_{{\by_x}\not\in R_t}\p\Big(\tau_{\bx-x\by_x}>\frac{x}{\hat{c}_1}-tx\Big)\to q.$$
In addition, by Cram\'{e}r's large deviation theorem,
$$\lim_{x\to\infty}\frac{1}{x}\log\p(\bS_{\lfloor tx\rfloor}\in xR_t)\leq -\inf_{\by_x\in \bar{R}_t}tI\big(\frac{\by_x}{t}\big),$$
where $\bar{R}_t$ is the closure of $R_t$.
Therefore, by \eqref{eq:ppp}, \eqref{eq:p2}, and Lemma \ref{lemma:GW LD}, we arrive at 
\begin{align*}
    \lim_{x\to\infty}\frac{1}{x}\log\p\Big(\tau_x>\frac{x}{\hat{c}_1}\mid S\Big)&\leq 2(\log\rho)\ee_2-\inf_{t\in F}\inf_{\by_x\in \bar{R}_t}\Big(tI\big(\frac{\by_x}{t}\big)+\gamma(t-\ee_2)\Big).
\end{align*}
We next show that
\begin{align}
    \lim_{\ee_3\to 0}\inf_{\by_x\in \bar{R}_t}\Big(tI\big(\frac{\by_x}{t}\big)+\gamma(t-\ee_2)\Big)= \inf_{\by_x:I((\be-\by_x)/(1/\hat{c}_1-t))\geq\log\rho}\Big(tI\big(\frac{\by_x}{t}\big)+\gamma (t-\ee_2)\Big).\label{eq:p3}
\end{align}
Indeed, the `$\leq$' direction is obvious; to see the `$\geq$' direction, let $\{\by_n\}$ be an arbitrary sequence such that $I((\be-\by_n)/(1/\hat{c}_1-t))\in(\log\rho-1/n,\log\rho]$ for each $n$. Since the rate function $I$ is good (Lemma 2.2.30 of \citep{dembo2009large} and the remark that follows), the level sets of $I$ are compact, so there is a subsequence $\by_{n_k}\to \by$. Since $I$ is continuous on its domain (which contains the level set), we must have $I((\be-\by)/(1/\hat{c}_1-t))=\log\rho$. This proves \eqref{eq:p3}. 
Therefore, using \eqref{eq:p3} and letting $\ee_3\to 0$ and then $\ee_2\to 0$, we have 
$$\lim_{x\to\infty}\frac{1}{x}\log\p\Big(\tau_x>\frac{x}{\hat{c}_1}\mid S\Big)\leq -\inf_{0\leq t\leq 1/\hat{c}_1}\inf_{\by_x:I((\be-\by_x)/(1/\hat{c}_1-t))\geq\log\rho}\Big(tI\big(\frac{\by_x}{t}\big)+\gamma t\Big),$$
as desired.
 \end{proof}

\section{An asymptotically logarithmically optimal algorithm for estimating the lower deviation probabilities}\label{sec:alg}

\subsection{Problem formulation}

Our goal in this section is to provide an efficient algorithm that estimates $\p(\tau_x\leq x/\hat{c}_1)$ where $\hat{c}_1>c_1$. We first argue that this reduces to finding one for $\p(\tau_x=\lfloor x/\hat{c}_1\rfloor)$. 
By Proposition \ref{prop:LB} and the proof of Theorem \ref{thm:Ld1}, for any $\hat{c}_1>c_1$ and $\ee>0$, there exists $K>0$ such that
$$\p\Big(\tau_x\leq  \frac{x}{\hat{c}_1}-K\Big)\leq \ee\p\Big(\tau_x=\lfloor\frac{x}{\hat{c}_1}\rfloor\Big)\leq \ee\p\Big(\tau_x\leq \frac{x}{\hat{c}_1}\Big).$$
In other words, to estimate $\p(\tau_x\leq x/\hat{c}_1)$ subject to an $\ee$-bias, it suffices to estimate $\p(x/\hat{c}_1-K<\tau_x\leq x/\hat{c}_1)$ for some large integer $K$. By writing 
$$\p\Big(\frac{x}{\hat{c}_1}-K< \tau_x\leq \frac{x}{\hat{c}_1}\Big)=\sum_{t=0}^{K-1}\p\Big(\tau_x=\lfloor\frac{x}{\hat{c}_1}\rfloor-t\Big),$$
it follows that estimating the left-hand side is equivalent to applying the estimation algorithm $K$ times for the probabilities of the form $\p(\tau_x= \lfloor x/\hat{c}_1\rfloor-t),~0\leq t\leq K-1$, where $K$ is a constant that does not increase with $x$. In other words, developing an algorithm for computing the cdf $\p(\tau_x\leq x/\hat{c}_1)$ with polynomial complexity is equivalent to developing one for $\p(\tau_x= \lfloor x/\hat{c}_1\rfloor-t)$ with polynomial complexity.

We impose slightly different assumptions from Section \ref{sec:LD}. In particular, for simplicity, we assume that $\bxi$ has a spherical symmetric law; i.e., $\bxi$ is invariant under any orthonormal transformation. The general non-symmetric case can be derived by a similar analysis. We summarize the assumptions as follows. Recall that $\phi_\bxi$ is the MGF of $\bxi$. 
 
\begin{itemize}
     \item [(A1')] the offspring distribution $\zeta $ has a finite $(1+\delta)$-th moment, i.e., $\sum_i i^{1+\delta}\,p_i<\infty$; 
     \item [(A2')] the law of $\bxi$ is non-degenerate, integrable, and spherically symmetric;
         \item [(A3')] $\hat{c}_1>c_1$ satisfies $(\hat{c}_1,\z)\in (\mathrm{ran}\nabla\log\phi_\bxi)^\circ$ and $\phi_\bxi$ is well-defined in an open neighborhood of $\hat{\bc}_2=\nabla I(\hat{c}_1,\z)$.
 \end{itemize}

Note that (A1') is weaker than (A1) and (A3') is weaker than (A3). 
The assumption (A2') implies the non-lattice condition (A4) by Lemma 24 of \citep{blanchet2024first}. 
Under assumptions (A2') and (A3'), we may abuse the notation and in this section denote by $I$ the large deviation rate function of $\xi$, the first coordinate of $\bxi$. We also write $\hat{c}_2=I'(\hat{c}_1)$, so that $\hat{\bc}_2=(\hat{c}_2,\z)$. In particular, (A3') implies that $\phi_\xi$ is well-defined in a neighborhood of $\hat{c}_2$, and $S_k$ defined in \eqref{eq:Sk} is the first coordinate of $\bS_k$. Again using spherical symmetry and (A3'), $\phi_\bxi$ is also well-defined in a neighborhood of $\z$.

The results in Sections \ref{sec:spineal} and \ref{sec:lowerd} were proved under assumptions (A1)--(A4), but it is direct to see that they also hold under (A1')--(A3').

\subsection{Description of the algorithm and main results}\label{sec:3.2}


Recall our notation for the spine decomposition from Section \ref{sec:spineal}. Let $n=\lfloor x/\hat{c}_1\rfloor-t$ be the desired value of the FPT, where $t=O(1)$. Our algorithm requires the introduction of a few parameters and events. Consider large positive constants $R_1,R_2,R_3,R_4,R_5$ to be determined. Eventually, we need $R_1\succ  R_2\succ  R_5\succ  R_4>d/(2\hat{c}_2)$ and $R_3\geq R_2$, where ``$\succ $'' here stands for ``much larger than''. 
Recall \eqref{eq:Dnk}, \eqref{eq:bar D}, and the definition of $\bar{c}_1$ above \eqref{eq:ccc}. Define the events
\begin{align}
    \begin{split}
        E_7&=\{\bS_n\in B_{R_1\log x}(\bx)\};\\
    E_8&=\{\exists k\in[n-R_2\log x,n],D_{n,k}(B_x)>0;\forall j\in[n-R_2\log x,n-1],\bar{D}_{n-1,j}(B_x)=0\};\\
    E_9&=\{S_n\leq x+R_4\log x\};\\
    E_{10}&=\{\forall k\in[n-\lfloor R_5\log x\rfloor],S_k<x+R_4\log x-\bar{c}_1(n-k)\};\\
    E_{11}&=\Big\{\forall k\in[n-\lfloor R_5\log x\rfloor],\sum_i e^{\hat{c}_2b_k(i)}<e^{\ee_1\hat{c}_2(n-k)}\Big\};\\
    E&=\bigcap_{j=7}^{11}E_j.
    \end{split}\label{eq:events}
\end{align}
 The events $E_7,\dots,E_{10}$ are introduced so that we can check that $\tau_x=n$ in polynomial (in $x$) computation time with a small relative bias under $\q$. The event $E_{11}$ allows us to compute the likelihood ratio with a small error in polynomial time in $x$.

\begin{figure}[!htbp]
    \centering
\includegraphics[width=0.97\linewidth,trim={2.5cm 1.8cm 2.2cm 1.5cm},clip]{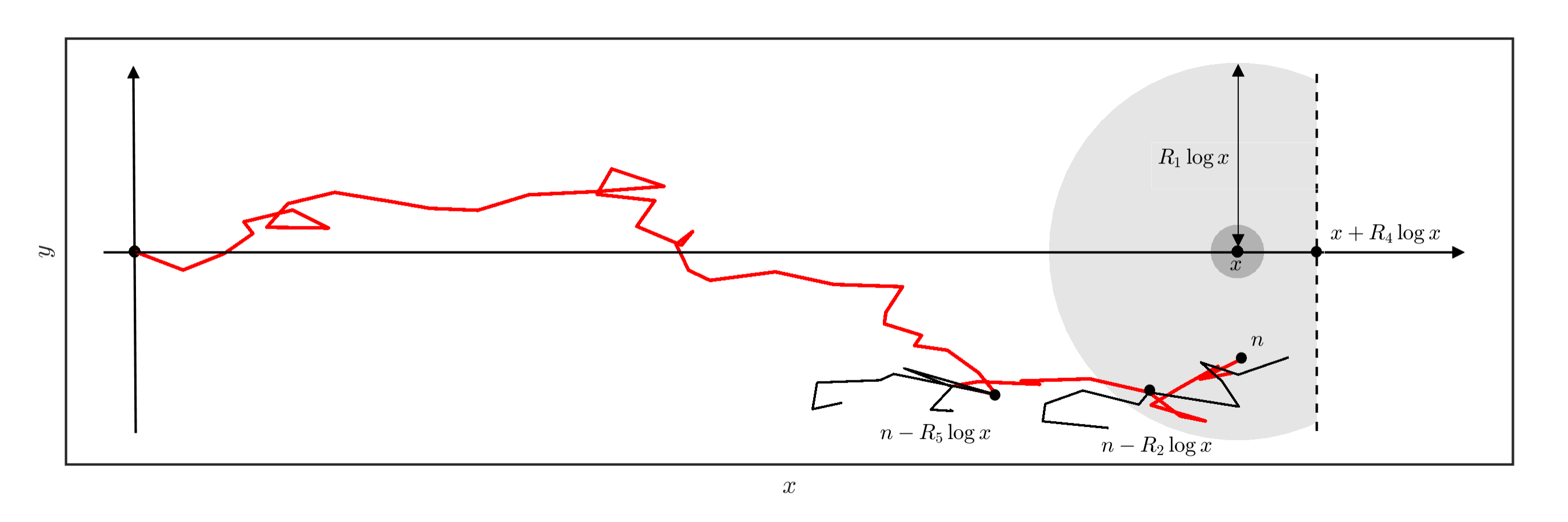}
    \caption{Schematic of the algorithm selecting a sample of a spine BRW, with $d=2$. The area shaded by light grey is the constrained location $\bS_n$ of the endpoint of the spine (event $E_7\cap E_9$). Further constraints are imposed on the spine trajectory (colored in red) and its adjacent jumps for steps in $[n-\lfloor R_5\log x\rfloor,n]$ (event $E_{10}\cap E_{11}$). If a sample verifies these constraints, the algorithm proceeds with simulating the off-spine BRWs and verifying the first passage event $\{\tau_x=n\}$ (event $E_8$).}
    \label{fig:schematicspine}
\end{figure}

\begin{remark}
    When checking if the event $E$ holds for the simulation under the law $\q$, one first simulates the spine $\{\bS_k\}$ and the bones $\{\bb_k(i)\}$ (which takes a linear amount of time), followed by the off-spine BRWs. The events $E_7,E_9,E_{10},E_{11}$ are $\sigma(\{\bS_k\},\{\bb_k(i)\})$-measurable. So, in the simulation procedure (detailed below), one checks that the event $E_7\cap E_9\cap E_{10}\cap E_{11}$ holds, before proceeding with the simulation of the off-spine BRWs. Intuitively, the events $E_7,E_8,E_9$ are less likely to hold than $E_{10}\cap E_{11}$ under law $\q$.  See also Figure \ref{fig:schematicspine} for an illustration.
\end{remark}


Let $R_3\geq R_2$ and define the set
\begin{align}
    W_n=\bigcup_{k\geq n-R_3\log x}V_{n,k}\subseteq V_n.\label{eq:Wn}
\end{align}
In other words, $W_n$ are particles at time $n$ that emanated from the spine during the last $R_3\log x$ steps.
Finally, let 
\begin{align}
    Z=Z(x):=\bigg(\sum_{v\in W_n}e^{\hat{c}_2 \eta_v-n(\hat{c}_1\hat{c}_2+\log\rho-I(\hat{c}_1))}\bigg)^{-1}\bone_E.\label{eq:Zdef}
\end{align}
The quantity inside the bracket serves as an approximation to $\d\q/\d\p$ (c.f.~\eqref{eq:dq/dp}). 

In what follows, we provide the pseudocode for estimating the lower deviation probability $\p(\tau_x= \lfloor x/\hat{c}_1\rfloor-t)$ where $t\in\N$ is fixed and $x$ is large enough. We call this algorithm the \textit{Trimmed-Spine-BRW algorithm}. As usual, we denote by $n=\lfloor x/\hat{c}_1\rfloor-t$.

\begin{tcolorbox}[title=Pseudocode for the Trimmed-Spine-BRW algorithm, colback=gray!5, colframe=black]

\begin{enumerate}
    \item Generate a sample of the spine $\{\bS_k\}$ and the bones $\{\bb_k(i)\}$, whose laws are specified in Section \ref{sec:spineal}. Verify that the event $E_7\cap E_9\cap E_{10}\cap E_{11}$ holds. If not, return a value zero. If so, proceed to the next step.
    \item Construct the spine BRW by generating next the off-spine BRWs that do not leave the spine before time $k=n-R_2\log x$. Verify that the event $E_8$ holds. If not, return a value zero. If so, proceed to the next step.
    \item Compute and return the value 
    $$ Z=\bigg(\sum_{v\in W_n}e^{\hat{c}_2 \eta_v-n(\hat{c}_1\hat{c}_2+\log\rho-I(\hat{c}_1))}\bigg)^{-1}\bone_E.$$
    \item Iterate the above steps independently $N$ times for a large $N$ and average the returned values to return the final estimate for $\p(\tau_x= \lfloor x/\hat{c}_1\rfloor-t)$.
\end{enumerate}
\end{tcolorbox}

\begin{remark}\label{rem:omega}
    The algorithm has several parameters, $R_j$ for $1\leq j\leq 5$. A recipe that is asymptotically logarithmically optimal and intuitive is to let $R_4=d/(2\hat{c}_2)$ and define the remaining parameters in terms of some $\omega>1$ as follows: $R_5 =\omega R_4$, $R_2=R_3=\omega^2R_4$, and $R_1=\omega^3R_4$. We implement the algorithm with these choices and compare their performance with varying $\omega$ in Section \ref{sec:numerical}.
\end{remark}

The following result provides a guarantee of the algorithm.
\begin{theorem}\label{prop:alg}
     Fix $\hat{c}_1>c_1$ and $t\in\N$. There exist $r>0$ such that the random variable $Z=Z(x)$ defined through \eqref{eq:Zdef} satisfies
     \begin{enumerate}[(i)]
         \item $0\leq Z\leq x^r\p(\tau_x= \lfloor x/\hat{c}_1\rfloor-t)$;\label{t1}
         \item $\E[Z]=(1+o(1))\p(\tau_x= \lfloor x/\hat{c}_1\rfloor-t)$ as $x\to\infty$;\label{t2}
         \item $\E[Z^2]\ll x^r\p(\tau_x= \lfloor x/\hat{c}_1\rfloor-t)^2$ as $x\to\infty$,\label{t3}
     \end{enumerate}
     and simulating $Z$ has a computational complexity upper bounded by $O(x^r)$.
\end{theorem}


\begin{remark}
    As will be apparent from the proof, the constant $r$ here may depend on the law of the BRW, and, in particular, depend on the underlying dimension $d$. Although we have not done so, $r$ can be made explicit.
\end{remark}


\begin{corollary}\label{thm:alg}
     Fix $\hat{c}_1>c_1$ and $t\in\N$. Let $Z$ be as defined in \eqref{eq:Zdef} and $\{Z_n\}_{n\in \N}$ be i.i.d.~copies of $Z$. Define also $\bar{Z}_N=(Z_1+\dots+Z_N)/N$.  For any $\ee>0$, there exists $r>0$ such that for all $x$ large enough,
\begin{align}
    \p\bigg(\Big|\p\Big(\tau_x= \lfloor \frac{x}{\hat{c}_1}\rfloor -t\Big)-\bar{Z}_N\Big|\geq \ee\p\Big(\tau_x= \lfloor \frac{x}{\hat{c}_1}\rfloor -t\Big)\bigg)\leq \exp\Big(-\frac{\ee^2N}{10x^r}\Big).\label{eq:toshow}
\end{align}
In particular, for any $\delta>0$, there exists $Z'$ such that
$$\p\bigg(\Big|\p\Big(\tau_x= \lfloor \frac{x}{\hat{c}_1}\rfloor -t\Big)-Z'\Big|\geq \ee\p\Big(\tau_x= \lfloor \frac{x}{\hat{c}_1}\rfloor -t\Big)\bigg)\leq\delta$$
 and simulating $Z'$ has a computational complexity upper bounded by $O(x^r(-\log\delta)/\ee^2)$.
\end{corollary}

\begin{proof}
     For $x$ large enough, by Theorem \ref{prop:alg}\eqref{t2}, $$|\E[Z]-\p(\tau_x= \lfloor x/\hat{c}_1\rfloor-t)|<\frac{\ee}{2}\p(\tau_x= \lfloor x/\hat{c}_1\rfloor-t).$$ Therefore, by Bernstein's inequality,
     \begin{align*}
         &\p\bigg(\Big|\p\Big(\tau_x= \lfloor \frac{x}{\hat{c}_1}\rfloor -t\Big)-\bar{Z}_N\Big|\geq\ee\p\Big(\tau_x= \lfloor \frac{x}{\hat{c}_1}\rfloor -t\Big)\bigg)
         \\&\leq \p\Big(|\bar{Z}_N-\E[\bar{Z}_N]|\geq\frac{\ee}{2}\p\Big(\tau_x=\lfloor\frac{x}{\hat{c}_1}\rfloor-t\Big)\Big)\\
         &\leq \exp\Big(-\frac{\big(\frac{\ee N}{2}\p(\tau_x= \lfloor x/\hat{c}_1\rfloor-t)\big)^2}{2(Nx^r\p(\tau_x= \lfloor x/\hat{c}_1\rfloor-t)^2+\frac{\ee}{6}Nx^r\p(\tau_x= \lfloor x/\hat{c}_1\rfloor-t)^2)}\Big)\\
         &\leq \exp\Big(-\frac{\ee^2N}{10x^r}\Big)
     \end{align*}
     for $x$ large enough. With the choice $N=\lceil (-10x^r\log\delta)/\ee^2\rceil$, we have that $Z'=\bar{Z}_N$ satisfies
     $$\p\bigg(\Big|\p\Big(\tau_x= \lfloor \frac{x}{\hat{c}_1}\rfloor -t\Big)-Z'\Big|\geq\ee\p\Big(\tau_x= \lfloor \frac{x}{\hat{c}_1}\rfloor -t\Big)\bigg)\leq \exp\Big(-\frac{\ee^2N}{10x^r}\Big)\leq \delta,$$
     as desired.
\end{proof}

Let us comment on the asymptotic optimality properties of our algorithm. Let $\hat{c}_1$ and $t$ be fixed. For each fixed $x$, the estimator $Z(x)$ in \eqref{eq:Zdef} is biased. However, Theorem \ref{prop:alg}\eqref{t2} indicates that the collection of estimators $\{Z(x)\}_{x\geq 0}$ is asymptotically unbiased as $x\to\infty$. Following \citep[Chapter VI]{asmussen2007stochastic} and \citep{blanchet2012state}, we say that $\{Z(x)\}_{x\geq 0}$ has \textit{bounded relative error} (or \textit{strongly efficient}) if 
$$\limsup_{x\to\infty}\frac{\var(Z(x))}{\E[Z(x)]^2}<\infty,$$
and is \textit{logarithmically efficient} (or \textit{asymptotically optimal}) if
\begin{align}
    \liminf_{x\to\infty}\frac{|\log\var(Z(x))|}{|\log\E[Z(x)]^2|}\geq 1.\label{eq:log-efficient}
\end{align}
It follows from Theorem \ref{prop:alg} and Proposition \ref{prop:LB} that $\{Z(x)\}_{x\geq 0}$ is logarithmically efficient, but does not have bounded relative error. In fact, $\{Z(x)\}_{x\geq 0}$ satisfies the following strengthened version of logarithmic efficiency (also known as having \textit{polynomial complexity of a finite order}):
$$\limsup_{x\to\infty}\frac{\var(Z(x))}{\E[Z(x)]^2|\log\E[Z(x)]^2|^r}<\infty$$
for some $r>0$.

Since the complexity of a BRW grows exponentially in time, it is unavoidable that bias will be introduced for the polynomial-complexity estimator. Polynomial complexity is achieved by trimming the simulation of the decorations (off-spine BRWs) in the spine decomposition until the last $L\log x$ steps, for some large constant $L$ to be determined. 

As logarithmic efficiency \eqref{eq:log-efficient} is a desirable property for estimators, it is essential that while trimming the decorations, the first moment $\E[Z(x)]$ remains of the same magnitude as $\p(\tau_x=n)$, which we have shown is of order $x^{-d/2}e^{-\frac{x}{\hat{c}_1}(I(\hat{\bc}_1)-\log\rho)}$ from Section \ref{sec:LD}. The careful construction of the events \eqref{eq:events} and the definition \eqref{eq:Zdef} achieves this goal.


\subsection{Proof of Theorem \ref{prop:alg}}\label{sec:proof}

Recall \eqref{eq:Zdef} and the events $E_7$ through $E_{11}$ defined in \eqref{eq:events}. 
We first verify some simple properties of $Z$. First, since $R_3\geq R_2$, on the event $E$ there exists $v\in W_n$ with $\bet_v\in B_x$ and hence $\eta_v\in[x-1,x+1]$. Therefore, on the event $E$, we have the trivial lower bound
\begin{align}
    \frac{\d\q}{\d\p}\geq \sum_{v\in W_n}e^{\hat{c}_2 \eta_v-n(\hat{c}_1\hat{c}_2+\log\rho-I(\hat{c}_1))}\gg e^{\hat{c}_2 x-n(\hat{c}_1\hat{c}_2+\log\rho-I(\hat{c}_1))}\asymp e^{n(I(\hat{c}_1)-\log\rho)}.\label{eq:LB trivial}
\end{align}
Therefore, $0\leq Z\ll e^{-n(I(\hat{c}_1)-\log\rho)}$ a.s. The same bound \eqref{eq:LB trivial} also holds on the event $\{\tau_x\leq n\}$. This along with Proposition \ref{prop:LB} gives Theorem \ref{prop:alg}\eqref{t1}. Immediately, Theorem \ref{prop:alg}\eqref{t3} also follows. Moreover, it is not hard to see that simulating $Z$ takes polynomial time in $x$, since it only concerns simulating a polynomial number of particles, each evolving for at most linear amount of time. We are therefore left with proving Theorem \ref{prop:alg}\eqref{t2}, that 
\begin{align}
    \E^\q[Z]=(1+o(1))\p\Big(\tau_x=\lfloor\frac{x}{\hat{c}_1}\rfloor-t\Big).\label{eq:ii}
\end{align}
The key part of the proof of \eqref{eq:ii} attributes to the next proposition.

\begin{proposition}\label{prop:final}
    There exist choices of positive numbers $R_1,R_2,R_3,R_4,R_5$ (recall from \eqref{eq:Wn} that the choice of $W_n$ depends on $R_3$) such that:
    \begin{enumerate}[(i)]
     
    \item $\p(E_9^c)=\p(S_n\geq x+R_4\log x)=o(x^{-d/2}e^{-n(I(\hat{c}_1)-\log\rho)})$;\label{PE7}
    \item $\q(E_9\cap E_{10}^c)=\q(S_n\leq x+R_4\log x;\exists k\in[n-\lfloor R_5\log x\rfloor],S_k\geq x+R_4\log x-\bar{c}_1(n-k))=o(x^{-d/2})$;\label{S1}
    \item $\q(E_9\cap E_{11}^c)=\q(S_n\leq x+R_4\log x;\exists k\in[n-\lfloor R_5\log x\rfloor],\sum_i e^{\hat{c}_2b_{k+1}(i)}\geq e^{\ee_1\hat{c}_2(n-k)})=o(x^{-d/2})$;\label{S2}
        \item $\q(\{\tau_x=n\}\cap E_7^c\cap E_9\cap E_{10}\cap E_{11})=o(x^{-d/2})$;\label{QE7}
       \item $\p(\tau_x\leq n;\exists k\in [n-R_2\log x], \bar{D}_{n,k}(B_x)>0)=o(x^{-d/2}e^{-n(I(\hat{c}_1)-\log\rho)})$;\label{PE8}
        \item it holds
        $$\E^\q\bigg[\sum_{v\in V_n\setminus W_n}e^{\hat{c}_2 \eta_v-n(\hat{c}_1\hat{c}_2+\log\rho-I(\hat{c}_1))}\bone_E\bigg]=o(x^{-d/2}e^{n(I(\hat{c}_1)-\log\rho)}).$$\label{QE}
    \end{enumerate}
\end{proposition}

\begin{proof}[Proof of Theorem \ref{prop:alg}] As argued above, it remains to prove \eqref{eq:ii}. Let $n=\lfloor x/\hat{c}_1\rfloor-t=x/\hat{c}_1+O(1)$. By the triangle inequality, \eqref{eq:Zdef}, and \eqref{eq:LB trivial}, we write
\begin{align*}
    &|\E^\q[Z]-\p(\tau_x=n)|\\
    &\leq \E^\q\bigg[\frac{\d\p}{\d\q}|\bone_{\tau_x=n}-\bone_E|\bigg]+\E^\q\Bigg[\bigg(\bigg(\sum_{v\in W_n}e^{\hat{c}_2 \eta_v-n(\hat{c}_1\hat{c}_2+\log\rho-I(\hat{c}_1))}\bigg)^{-1}-\frac{\d\p}{\d\q}\bigg)\bone_E\Bigg]\\
    &\leq \p(E_9^{\mathrm{c}})+e^{-n(I(\hat{c}_1)-\log\rho)}\q(\{\tau_x=n\}\cap E_7^{\mathrm{c}}\cap E_9\cap E_{10}\cap E_{11})\\
    &\hspace{2cm}+e^{-n(I(\hat{c}_1)-\log\rho)}(\q(E_9\cap E_{10}^{\mathrm{c}})+\q(E_9\cap E_{11}^{\mathrm{c}}))\\
    &\hspace{2cm}+\p(\{\tau_x=n\}\cap E_8^{\mathrm{c}})+\p(\{\tau_x=n\}^{\mathrm{c}}\cap E_8)\\
    &\hspace{2cm}+\E^\q\Bigg[e^{-2n(I(\hat{c}_1)-\log\rho)}\sum_{v\in V_n\setminus W_n}e^{\hat{c}_2 \eta_v-n(\hat{c}_1\hat{c}_2+\log\rho-I(\hat{c}_1))}\bone_E\Bigg].
\end{align*}
\sloppy By Proposition \ref{prop:final}(\ref{QE7}), $\q(\{\tau_x=n\}\cap E_7^{\mathrm{c}}\cap E_9\cap E_{10}\cap E_{11})=o(x^{-d/2})$; by Proposition \ref{prop:final}(\ref{PE7}), $\p(E_9^{\mathrm{c}})=o(x^{-d/2}e^{-n(I(\hat{c}_1)-\log\rho)})$; by Proposition \ref{prop:final}(\ref{S1})\&(\ref{S2}), $\q(E_9\cap E_{10}^{\mathrm{c}})+\q(E_9\cap E_{11}^{\mathrm{c}})=o(x^{-d/2})$. 
Note also that
$$\{\tau_x=n\}\cap E_8^{\mathrm{c}}\subseteq \{\tau_x=n\}\cap\{\exists k\in[n-R_2\log x], \bar{D}_{n,k}(B_x)>0\}$$
and that since $E_8\subseteq\{\tau_x\leq n\}$, 
$$\{\tau_x=n\}^{\mathrm{c}}\cap E_8\subseteq \{\tau_x<n\}\cap E_8\subseteq \{\tau_x<n\}\cap\{\exists k\in[n-R_2\log x], \bar{D}_{n,k}(B_x)>0\},$$
so by Proposition \ref{prop:final}(\ref{PE8}), $\p(\{\tau_x=n\}\cap E_8^{\mathrm{c}})+\p(\{\tau_x=n\}^{\mathrm{c}}\cap E_8)=o(x^{-d/2}e^{-n(I(\hat{c}_1)-\log\rho)})$. 
Finally, by Proposition \ref{prop:final}(\ref{QE}), 
$$\E^\q\Bigg[e^{-2n(I(\hat{c}_1)-\log\rho)}\sum_{v\in V_n\setminus W_n}e^{\hat{c}_2 \eta_v-n(\hat{c}_1\hat{c}_2+\log\rho-I(\hat{c}_1))}\bone_E\Bigg]=o(x^{-d/2}e^{-n(I(\hat{c}_1)-\log\rho)}).$$

Altogether, we conclude using also Proposition \ref{prop:LB} that
$$|\E^\q[Z]-\p(\tau_x=n)|=o(x^{-d/2}e^{-n(I(\hat{c}_1)-\log\rho)})=o(\p(\tau_x= \lfloor x/\hat{c}_1\rfloor-t)),$$
    as desired.
\end{proof}

We start the proof of Proposition \ref{prop:final} with the following two technical lemmas that remove certain unlikely events.

\begin{lemma}\label{lemma:2balls}
    Fix $R_2,R_4>0$, then for any $R_1$ large enough, the following holds. Let $A_x=((-\infty,x+R_4\log x]\times\R^{d-1})\setminus B_{R_1\log x}(\bx)$. Then 
    \begin{align}
        \q(\exists k\in[n], \bS_k\in B_{R_2\log x}(\bx); \bS_n\in A_x)=o(x^{-d/2}).\label{eq:2balls}
    \end{align}
\end{lemma}

\begin{proof}
Consider some $R_6>0$ to be determined, and we split the index $k$ into the two regions $k\in[1,n-R_6\log x]$ and $k\in (n-R_6\log x,n]$. In the first case, we have that the event on the left-hand side of \eqref{eq:2balls} implies $S_n-S_k\leq (R_2+R_4)\log x$; in the second case, the event implies $\n{\bS_n-\bS_k}\geq (R_1-R_2)\log x$. 
Therefore, by the union bound and the stationarity of the increments, we obtain
\begin{align*}
    &\q(\exists k\in[n], \bS_k\in B_{R_2\log x}(\bx); \bS_n\in A_x)\\
    &\leq \q(\exists k\in[1,n-R_6\log x], \bS_k\in B_{R_2\log x}(\bx); \bS_n\in A_x)\\
    &\hspace{2cm}+\q(\exists k\in[n-R_6\log x, n], \bS_k\in B_{R_2\log x}(\bx); \bS_n\in A_x)\\
    &\leq \sum_{k=1}^{\lfloor n-R_6\log x\rfloor} \q(S_{n-k}\leq (R_2+R_4)\log x)+\sum_{k=\lfloor n-R_6\log x\rfloor+1}^{n}\q(\n{\bS_{n-k}}\geq (R_1-R_2)\log x).
\end{align*}
Note that the random walk $\{S_k\}$ has a positive drift $\hat{c}_1$ under $\q$. Using Cram\'{e}r's large deviation upper bound, we have that for $R_6$ large enough compared to $R_2+R_4$ but fixed, there is a uniform constant $\delta_7>0$ such that $\q(S_{n-k}\leq (R_2+R_4)\log x)\leq e^{-(n-k)\delta_7}$ for $k\in[1,n-R_6\log x]$. Similarly, using the triangle inequality, we have for $R_1-R_2$ large enough compared to $R_6$, $\q(\n{\bS_{n-k}}\geq (R_1-R_2)\log x)\ll e^{-\delta_8 R_6\log x}$ for $k\in(n-R_6\log x,n]$ (note that $\delta_8$ does not depend on $R_6$ as long as $(R_1-R_2)/R_6$ is large enough). Altogether, we have
\begin{align*}
    \q(\exists k\in[n], \bS_k\in B_{R_2\log x}(\bx); \bS_n\in A_x)&\ll \sum_{k=1}^{\lfloor n-R_6\log x\rfloor} e^{-(n-k)\delta_7}+\sum_{k=\lfloor n-R_6\log x\rfloor+1}^{n}e^{-\delta_8 R_6\log x}\\
    &\ll x^{-\min\{\delta_7,\delta_8/2\}R_6}.
\end{align*}
For $R_6$ picked large enough, we have $x^{-\min\{\delta_7,\delta_8/2\}R_6}=o(x^{-d/2})$, as desired.    
\end{proof}

\begin{lemma}\label{lemma:R2R5}
Fix $R_5>0$. Then for $R_2>0$ large enough,
    $$\q\bigg(\exists v\in \bigcup_{n-\lfloor R_5\log x\rfloor\leq k\leq n}V_{n,k},\bet_v\in B_x; \forall k\in[n], \bS_k\not\in B_{R_2\log x}(\bx)\bigg)=o(x^{-d/2}).$$
\end{lemma}

\begin{proof}
    The proof follows a line similar to that of Lemma \ref{lemma:2balls}, using union bound and large deviation arguments. 
    Using (A2'), (A3'), and the same arguments that lead to \eqref{eq:1+delta condition} below, one can show that for some $\delta_9>0$, $\E^\q[\sum_i e^{\delta_9\n{\bb_\ell(i)}}]<\infty$. Also note that the sum over $i$ is finite and the number of terms is given by the size-biased law of $\zeta$. Therefore, following large deviation estimates, one can show that for $R_2$ much larger than $R_5$,

    $$\sum_{k=n-\lfloor R_5\log x\rfloor}^n\sum_{i\neq w_{k+1}}\sum_{u\in V^{(i,k)}_{n-k-1}}\q\Big(\n{\bb_{k+1}(i)+\bet_u^{(i,k)}}>R_2\log x\Big)=o(x^{-d/2}).$$
    The desired result then follows from the union bound.
\end{proof}




\begin{proof}[Proof of Proposition \ref{prop:final}]

\eqref{PE7} 
 Let $M_n$ denote the maximum among the first coordinates of all particles' locations at time $n$. 
By definition, 
\begin{align}
    \p(S_n\geq x+R_4\log x)\leq \p(M_n\geq x+R_4\log x).\label{eq:SnMn}
\end{align}
By the first moment method and Cram\'{e}r's upper bound, we have for $R_4$ picked large enough,
\begin{align*}
    \p(M_n\geq x+R_4\log x)&\leq \rho^n\p\Big(\sum_{i=1}^n\xi_i\geq x+R_4\log x\Big)\\
    &\leq \rho^n e^{-nI(\hat{c}_1)-I'(\hat{c}_1)R_4\log x}=o(x^{-d/2}e^{-n(I(\hat{c}_1)-\log\rho)}).
\end{align*}

\eqref{S1} We let $R_5\to\infty$. Recall from above \eqref{eq:ccc} that $\bar{c}_1\in(c_1,\hat{c}_1)$. A standard large deviation upper bound shows that for some $\delta_{10}>0$,
\begin{align*}
   &\q(\exists k\in[n-\lfloor R_5\log x\rfloor],S_k\geq x+R_4\log x-\bar{c}_1(n-k); S_n\leq x+R_4\log x)\\
    &\leq \sum_{\ell=\lfloor R_5\log x\rfloor}^n \q(S_\ell\leq \bar{c}_1\ell)\ll \sum_{\ell=\lfloor R_5\log x\rfloor}^n e^{-\delta_{10}\ell}\ll x^{-\delta_{10} R_5}.
\end{align*}
With $R_5$ chosen large enough, this is $o(x^{-d/2})$.

\eqref{S2}  By Lemma \ref{lemma:b_k} applied with $\ell_0=\lfloor R_5\log x\rfloor$, we know that outside an event with probability $o(x^{-d/2})$ and for $R_5$ large enough, for each $k\in[n-\lfloor R_5\log x\rfloor]$, $\sum_i e^{\hat{c}_2b_{k+1}(i)}<e^{\ee_1\hat{c}_2(n-k)}$. 

(\ref{QE7}) Using the notations of Lemma \ref{lemma:2balls}, we have
\begin{align*}
    &\q(\{\tau_x=n\}\cap E_7^c\cap E_9\cap E_{10}\cap E_{11})\\
    &=\q(\tau_x=n; \bS_n\in A_x;E_{10}\cap E_{11})\\
    &\leq o(x^{-d/2})+\q(\tau_x=n; \bS_n\in A_x; \forall k\in[n], \bS_k\not\in B_{R_2\log x}(\bx);E_{10}\cap E_{11}).
\end{align*}
By definition, we have
\begin{align*}
    &\q(\tau_x=n; \bS_n\in A_x; \forall k\in[n], \bS_k\not\in B_{R_2\log x}(\bx);E_{10}\cap E_{11})\\
    &\leq \q(\exists v\in V_n,\bet_v\in B_x; S_n\leq x+R_4\log x; \forall k\in[n], \bS_k\not\in B_{R_2\log x}(\bx);E_{10}\cap E_{11}).
\end{align*}
By Lemma \ref{lemma:R2R5}, for $R_2$ chosen large enough in terms of $R_5$,
$$\q\bigg(\exists v\in \bigcup_{n-\lfloor R_5\log x\rfloor\leq k\leq n}V_{n,k},\bet_v\in B_x; S_n\leq x+R_4\log x; \forall k\in[n], \bS_k\not\in B_{R_2\log x}(\bx)\bigg)=o(x^{-d/2}).$$
    Therefore, by the union bound, it remains to show 
    \begin{align}
        \q\bigg(\exists v\in \bigcup_{k=0}^{n-\lfloor R_5\log x\rfloor}V_{n,k}, \bet_v\in B_x; E_9\cap E_{10}\cap E_{11}\bigg)=o(x^{-d/2}).\label{eq:s1}
    \end{align}
    We then establish \eqref{eq:s1} by showing that, with $R_5$ overwhelmingly larger than $R_4$, the sum of exponential moments over the relevant particles is small conditioned on $\{\bS_k\},\{\bb_k(i)\}$, through the same argument in the proof of Proposition \ref{prop:LB}. 
    Indeed, we have by \eqref{eq:psibth} and \eqref{eq:ccc},
\begin{align}
    \begin{split}
        &\E^\q\bigg[\sum_{k=0}^{n-\lfloor R_5\log x\rfloor}\sum_{v\in V_{n,k}}e^{\hat{c}_2\eta_v}\mid E_9\cap E_{10}\cap E_{11},\{\bS_k\},\{\bb_k(i)\}\bigg]\\
    &=\bone_{E_9\cap E_{10}\cap E_{11}}\sum_{k=0}^{n-\lfloor R_5\log x\rfloor} e^{\hat{c}_2 S_k}\bigg(\sum_{i=1}^{N_k} e^{\hat{c}_2 b_{k+1}(i)}\bigg)\E^\q\bigg[\sum_{v\in V^{(i,k)}_{n-k-1}}e^{\hat{c}_2\eta^{(i,k)}_v}\bigg]\\
    &\ll \sum_{k=0}^{n-\lfloor R_5\log x\rfloor} e^{\hat{c}_2(x+R_4\log x-(n-k)\bar{c}_1)}e^{\ee_1\hat{c}_2(n-k)} e^{(n-k-1)(\hat{c}_1\hat{c}_2+\log\rho-I(\hat{c}_1))}\\
    &\ll e^{\hat{c}_2 x}x^{-\hat{c}_2(\ee_1 R_5-R_4)}.
    \end{split}
    \label{eq:s2}
\end{align} It follows that for $R_5$ chosen large enough, \eqref{eq:s1} holds.    

\eqref{PE8} This follows from a similar moment computation above in the proof of \eqref{QE7}, where we proved the stronger fact that 
\begin{align}
    \q(S_n\leq x+R_4\log x;\exists k\in[n-\lfloor R_5\log x\rfloor],D_{n,k}([x,\infty)\times \R^{d-1})=0)=o(x^{-d/2}).\label{eq:D bar D}
\end{align}
To see that the $D_{n,k}$ can be replaced by $\bar{D}_{n,k}$ in \eqref{eq:D bar D}, simply note that the sum in \eqref{eq:s2} can be replaced by a sum over all past trajectories without changing the final result:
\begin{align*}
    &\E^\q\bigg[\sum_{k=0}^{n-\lfloor R_5\log x\rfloor}\sum_{v\in \bar{V}_{n,k}}e^{\hat{c}_2\eta_v}\mid E_9\cap E_{10}\cap E_{11},\{\bS_k\},\{\bb_k(i)\}\bigg]\\
    &=\bone_{E_9\cap E_{10}\cap E_{11}}\sum_{k=0}^{n-\lfloor R_5\log x\rfloor} e^{\hat{c}_2 S_k}\bigg(\sum_{i=1}^{N_k} e^{\hat{c}_2 b_{k+1}(i)}\bigg)\sum_{j=0}^{n-k-1}\E^\q\bigg[\sum_{v\in V^{(i,k)}_{j}}e^{\hat{c}_2\eta^{(i,k)}_v}\bigg]\\
    &\ll \sum_{k=0}^{n-\lfloor R_5\log x\rfloor} e^{\hat{c}_2(x+R_4\log x-(n-k)\bar{c}_1)}e^{\ee_1\hat{c}_2(n-k)} e^{(n-k-1)(\hat{c}_1\hat{c}_2+\log\rho-I(\hat{c}_1))}\\
    &\ll e^{\hat{c}_2 x}x^{-\hat{c}_2(\ee_1 R_5-R_4)}.
\end{align*}
Finally, applying \eqref{eq:LB trivial} (which holds on the event $\{\tau_x=n\}\cap E_8$) turns this $\q$-probability into a $\p$-probability.

\eqref{QE} Recall the definition \eqref{eq:Wn} and that $R_3$ can be chosen much larger than $R_4$. It remains to show the stronger statement that
\begin{align}
    \E^\q\bigg[\sum_{v\in V_n\setminus W_n}e^{\hat{c}_2(\eta_v-x)}\bone_{E_9\cap E_{10}\cap E_{11}}\bigg]=o(x^{-d/2}).\label{eq:s3}
\end{align}
But this follows immediately from the previous computation \eqref{eq:s2}.
\end{proof}



\section{Numerical experiments}\label{sec:numerical}


The numerical experiments presented in this section are designed to validate the mathematical proofs presented in the previous sections. They allow us to verify the asymptotic predictions and scaling laws governing both the estimator's performance and the underlying probabilistic estimates in numerically tractable regimes. Furthermore, we investigate the influence of various parameters on the estimator analyzed in the previous sections. While the asymptotic optimality of the estimator was derived for large values of $x$ (with bounds that hold uniformly over $x$), our aim here is to study the behavior of the estimator for moderate values of $x$, which are more tractable in practical settings. 
We systematically list the parameters that appear in the design of the estimator and discuss how sensitive the running time and mean squared error of the algorithm are to these choices. Some of these parameters implicitly induce a bias-variance trade-off in the overall estimator because reducing the bias within a fixed computational budget may necessitate a longer per-run time, thereby affecting the overall efficiency. Although some of our theoretical results require taking certain parameters to be large for technical reasons, the experiments demonstrate that, in practice, these parameters can be chosen to be reasonably small without affecting the bias significantly, yielding substantial improvements in computational speed.

Throughout this section, we set $d=3$ and assume that $p_3 = 0.0856$ and $p_1 = 1-p_3$, so that $\p(S)=1$ and $\rho=1.1712$. The jump distribution is the uniform distribution on $\S^{d-1}$, the $(d-1)$-dimensional sphere, which is spherically symmetric. Setting $d=3$ gives $c_1=0.319$. 
 We also use the short-hand notation $n=\lfloor x/\hat{c}_1\rfloor$.

\subsection{Comparison of theory against  brute-force Monte Carlo BRW}
In this section, we present a comparison of the large deviation probabilities described in Theorems~\ref{thm:Ld1} and \ref{thm:LD2} against the FPT distribution obtained from $10^6$ samples generated using brute-force Monte Carlo BRW simulations. Figure~\ref{fig:fpt_dist_ref_theory} shows that both the lower-tail and upper-tail probabilities predicted by Theorems~\ref{thm:Ld1} and \ref{thm:LD2} agree well with the extremal probability distribution of the FPT obtained from the brute-force Monte Carlo BRW simulations. It is important to note that the lower-tail and upper-tail probabilities derived from Theorems~\ref{thm:Ld1} and \ref{thm:LD2} may deviate by at least a proportionality constant. The theoretical predictions shown in Figure~\ref{fig:fpt_dist_ref_theory_a} are given by $x^{-d/2}
    \exp(-\frac{x}{\hat{c}_1}(I(\hat{c}_1,\z)-\log\rho))$ for the lower tail and $\exp(-xT(x,\hat{c}_1;\rho,\gamma))$ for the upper tail, both scaled using the average proportionality constant (i.e., the scaling ratio) across different large deviation rates.  For the present case of $p_3 = 0.0856$, $p_1 = 1-p_3$, and $x=100$, the average scaling constants are $0.552$ for the lower tail and $8.318$ for the upper tail. This scaling is applied to facilitate comparison with brute-force Monte Carlo BRW simulations at the FPTs corresponding to the tails of the distribution. However, in practice these proportionality constants depend on $\hat{c}_1$ as described in Theorem \ref{thm:Ld1}. This dependence is illustrated in Figure~\ref{fig:fpt_dist_ref_theory_b}. 
\begin{figure}[!htbp]
    \centering
    \subfigure[]{\includegraphics[width=0.48\linewidth]{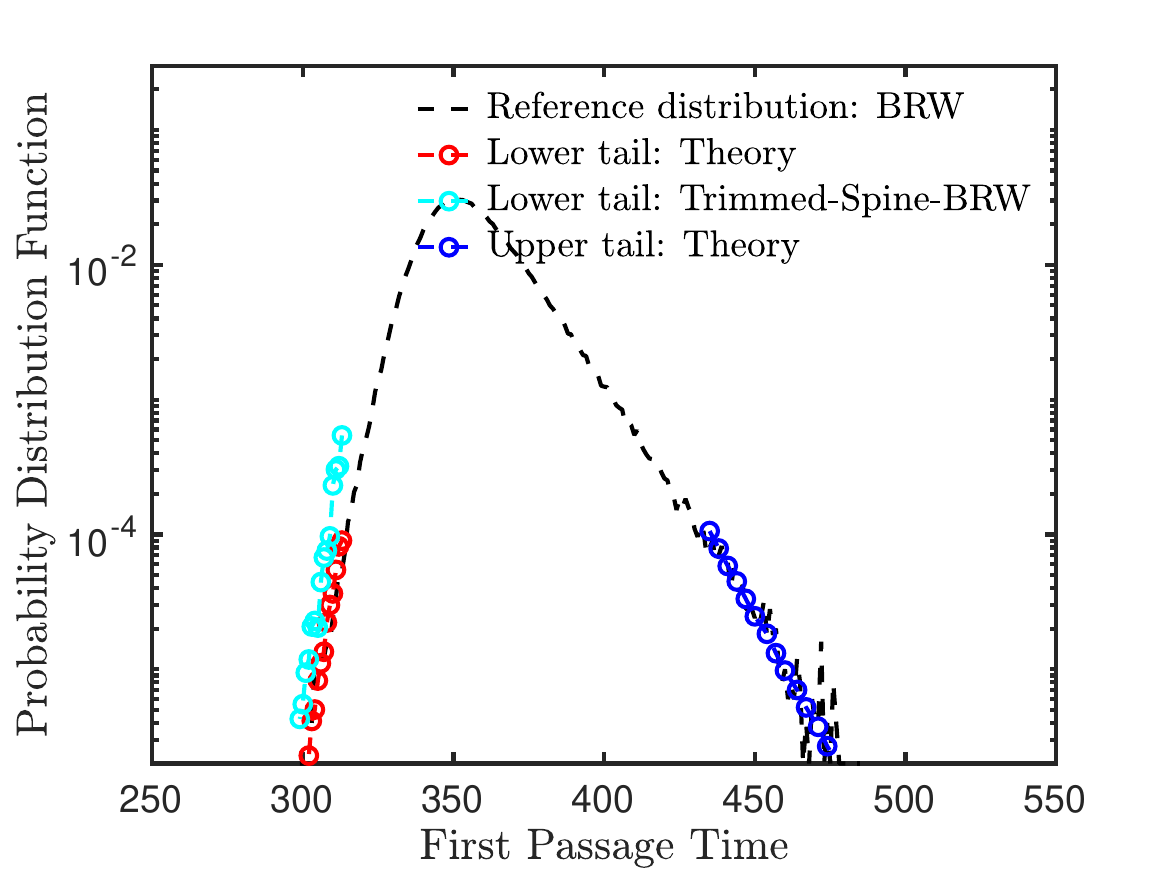}\label{fig:fpt_dist_ref_theory_a}}
    \subfigure[]{\includegraphics[width=0.48\linewidth]{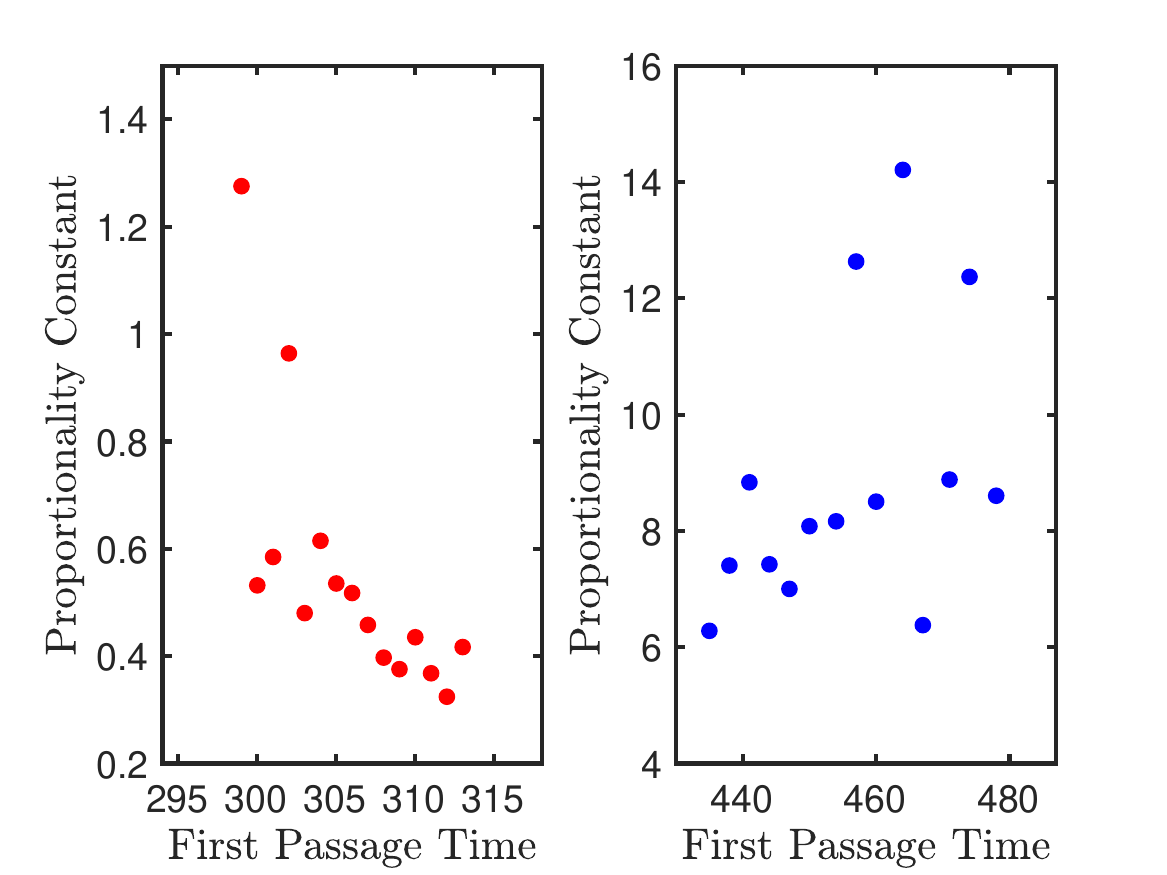}\label{fig:fpt_dist_ref_theory_b}}
    \caption{(a) Comparing the probability mass functions of the FPT between the brute-force Monte Carlo BRW generated from $10^6$ samples (dashed line), the scaled prediction from Theorems~\ref{thm:Ld1} and \ref{thm:LD2} (red circular-dashed plot for the lower tail and the blue circular-dashed plot for the upper tail) and our importance sampling algorithm (spine BRW; cyan circular-dashed plot) for $p_3 = 0.0856$, $p_1 = 1-p_3$, $x=100$, and $\hat{c}_1 \in [1.001c_1,\,1.02c_1]$ for the lower tail and $\hat{c}_1 \in [0.65c_1,\,0.72c_1]$ for the upper tail. (b) Proportionality constants required to match the theoretical predictions with brute-force Monte Carlo BRW results at the FPT corresponding to the lower and upper tails of the distribution. 
    \label{fig:fpt_dist_ref_theory}}
\end{figure}
%

%

%
\subsection{Comparison of the Trimmed-Spine-BRW algorithm against brute-force Monte Carlo}
In this section, we present the numerical implementation of the Trimmed-Spine-BRW algorithm, described in Section~\ref{sec:3.2}, including the details of its performance compared to brute-force Monte Carlo BRW simulations. A key advantage of the Trimmed-Spine-BRW algorithm is its ability to efficiently generate extremely rare paths (on the lower tail)—beyond the reach of brute-force BRW simulations—enabling the estimation of probabilities as small as $10^{-42}$, as shown later in Figure~\ref{fig:CDF_c1}.

By Theorem \ref{thm:Ld1}, the approximation of $\p(\tau_x=n)$ is given by 
\begin{align}
    \p(\tau_x=n) \asymp n^{-3/2} {\exp}\left(-\frac{x}{\hat{c}_1}\left(I(\hat{c}_1)-\log \rho\right)\right),\label{eq:th}
\end{align}
 Figure~\ref{fig:CDF_c1} shows good agreement in the numerically estimated values of $\p(\tau_x=n)$ with the theoretical asymptotics \eqref{eq:th} derived from Theorem \ref{thm:Ld1} (which is up to a multiplicative constant). Here we use $10^5$ samples, $\hat{c}_1\in[1.02\,c_1,\,1.5\,c_1]$ for $x=100$ in the Trimmed-Spine-BRW simulations. 

%
\begin{figure}[!htbp]
    \centering  
    \includegraphics[width=0.48\linewidth]{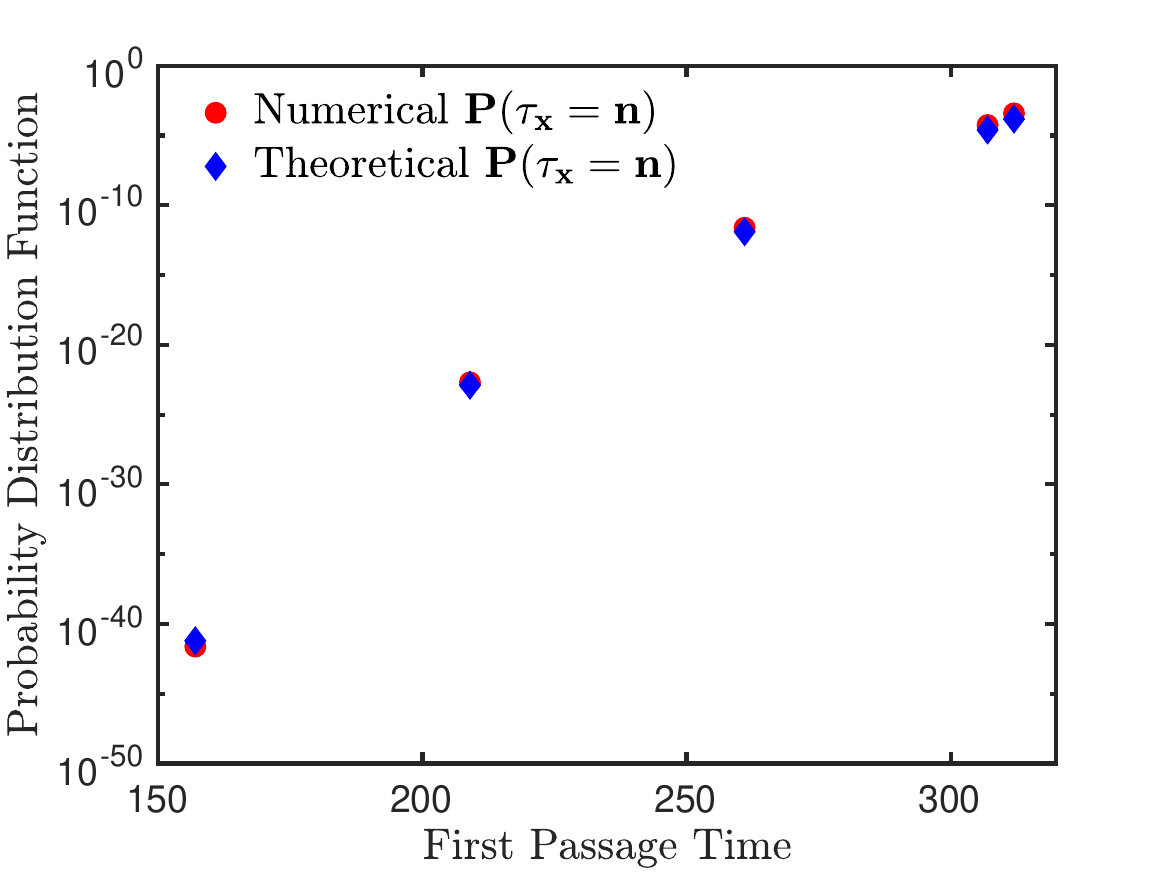}
    \caption{Comparison of $\p(\tau_x=n)$  between the Trimmed-Spine-BRW simulation and the theoretical asymptotics \eqref{eq:th}, as a function of the effective FPT with $10^5$ samples, $\hat{c}_1\in[1.02\,c_1,\,1.5\,c_1]$ for  $x=100$. 
    }
    \label{fig:CDF_c1}
\end{figure}
%

In the next experiment, we compare the probability mass function of the FPT estimated using the Trimmed-Spine-BRW algorithm against the brute-force Monte Carlo BRW implementation with $p_3 = 0.0856$ and $p_1 = 1-p_3$. The numerical results confirm that the Trimmed-Spine-BRW algorithm successfully captures the lower tail of the FPT distribution in the range $\hat{c}_1 \in [1.001c_1,\,1.02c_1]$, as shown in Figure~\ref{fig:fpt_dist_ref_theory_a} (cyan circular-dashed plot).
This agreement indicates that our method provides an efficient and reliable alternative to brute-force Monte Carlo simulations, particularly for rare-event estimation, where we can capture the extremal values of the FPT distribution by increasing $\hat{c}_1$.
%
%

Finally, to further investigate the dependence of the exponent, $$-\lim_{x\to\infty} \frac{\log\p(\tau_x=n)+\frac{x}{\hat{c}_1}(I(\hat{c}_1)-\log \rho)}{\log n}$$ in the analytic estimate of $\p(\tau_x=n)$, we perform a one-dimensional parameter scan over $\omega$ for the parameters $R_1,R_2,R_3,R_4,R_5$ in Section \ref{sec:3.2} (see Remark \ref{rem:omega}).
Note that the theoretical value of the exponent is $d/2=3/2$. 
Specifically, we set $R_4=d/(2\hat{c}_2)$, and define the remaining parameters in terms of $\omega$ as follows: $R_5 =\omega R_4$, $R_2=R_3=\omega^2R_4$, and $R_1=\omega^3R_4$ (which aligns with the order these parameters are chosen; see the beginning of Section \ref{sec:3.2}). We examine the scaling behavior obtained from the Trimmed-Spine-BRW simulations for a range of values $\omega \in[1,2]$, and observe that for $\omega>1$ the exponent of $3/2$ is captured within $1\%$. 

\section{Conclusions}
\label{sec:conclusions}
This paper focuses on understanding and simulating the first passage times of branching random walks in $\R^d$ in the large deviation regime. We derive the asymptotic order of the lower-tail large deviation probabilities and show that the upper-tail large deviation rates are connected to solutions to an optimization problem involving large deviation rate functions. Moreover, we provide a polynomial-time algorithm, called the Trimmed-Spine-BRW algorithm, that estimates the lower-tail large deviation probabilities. We confirm that our algorithm is asymptotically unbiased and logarithmically efficient.

The techniques in this work extend to many other relevant settings. For example, the same ideas extend to understanding the upper large deviation of the maximum of a one-dimensional BRW. In this case, \citep[Theorem 1.3]{luo2025precise} proves precise estimates of the large deviation probability, but the asymptotic constant remains implicit. Our approach leads to an asymptotically optimal algorithm that computes the asymptotic constant numerically.

\section*{Acknowledgement}

The material in this paper is partly supported by the Air Force Office of Scientific Research under award number FA9550-20-1-0397 and ONR N000142412655. Support from NSF 2229012, 2312204, 2403007 is also gratefully acknowledged. The open source code for the Trimmed-Spine-BRW algorithm can be accessed from \href{https://gitlab.com/micronano_public/PolyBranchX}{PolyBranchX}. 

\bibliographystyle{plain} 
\bibliography{reference}  

\begin{thebibliography}{10}

\bibitem{addario2009minima}
Louigi Addario-Berry and Bruce Reed.
\newblock Minima in branching random walks.
\newblock {\em Annals of Probability}, 37(3):1044--1079, 2009.

\bibitem{aidekon2013branching}
Elie A{\"i}d{\'e}kon, Julien Berestycki, {\'E}ric Brunet, and Zhan Shi.
\newblock Branching {B}rownian motion seen from its tip.
\newblock {\em Probability Theory and Related Fields}, 157:405--451, 2013.

\bibitem{asmussen2007stochastic}
S{\o}ren Asmussen and Peter~W Glynn.
\newblock {\em Stochastic Simulation: Algorithms and Analysis}, volume~57.
\newblock Springer, 2007.

\bibitem{athreya2004branching}
Krishna~B Athreya and Peter~E Ney.
\newblock {\em Branching {P}rocesses}.
\newblock Courier Corporation, 2004.

\bibitem{bansaye2013lower}
Vincent Bansaye and Christian B{\"o}inghoff.
\newblock Lower large deviations for supercritical branching processes in random environment.
\newblock {\em Proceedings of the Steklov Institute of Mathematics}, 282:15--34, 2013.

\bibitem{basrak2022importance}
Bojan Basrak, Michael Conroy, Mariana Olvera-Cravioto, and Zbigniew Palmowski.
\newblock Importance sampling for maxima on trees.
\newblock {\em Stochastic Processes and their Applications}, 148:139--179, 2022.

\bibitem{berestycki2024extremal}
Julien Berestycki, Yujin~H Kim, Eyal Lubetzky, Bastien Mallein, and Ofer Zeitouni.
\newblock The extremal point process of branching {B}rownian motion in {$\mathbb{R}^d$}.
\newblock {\em Annals of Probability}, 52(3):955--982, 2024.

\bibitem{bezborodov2023maximal}
Viktor Bezborodov and Nina Gantert.
\newblock The maximal displacement of radially symmetric branching random walk in $\mathbb{R}^d$.
\newblock {\em arXiv preprint arXiv:2309.14738}, 2023.

\bibitem{blanchet2024first}
Jose Blanchet, Wei Cai, Shaswat Mohanty, and Zhenyuan Zhang.
\newblock On the first passage times of branching random walks in $\mathbb{R}^{d}$.
\newblock {\em arXiv preprint arXiv:2404.09064}, 2024.

\bibitem{blanchet2012state}
Jose Blanchet and Henry Lam.
\newblock State-dependent importance sampling for rare-event simulation: An overview and recent advances.
\newblock {\em Surveys in Operations Research and Management Science}, 17(1):38--59, 2012.

\bibitem{blanchet2024tightness}
Jose Blanchet and Zhenyuan Zhang.
\newblock Tightness analysis of first passage times of $ d $-dimensional branching random walk.
\newblock {\em arXiv preprint arXiv:2410.02635}, 2024.

\bibitem{bramson2016convergence}
Maury Bramson, Jian Ding, and Ofer Zeitouni.
\newblock Convergence in law of the maximum of nonlattice branching random walk.
\newblock {\em Annales de l'Institut Henri Poincar\'{e} - Probabilit\'{e}s et Statistiques}, 52(4):1897--1924, 2016.

\bibitem{brunet2020generate}
{\'E}ric Brunet, Anh~Dung Le, Alfred~H Mueller, and St{\'e}phane Munier.
\newblock How to generate the tip of branching random walks evolved to large times.
\newblock {\em Europhysics Letters}, 131(4):40002, 2020.

\bibitem{buraczewski2019large}
Dariusz Buraczewski and Mariusz Ma{\'s}lanka.
\newblock Large deviation estimates for branching random walks.
\newblock {\em ESAIM: Probability and Statistics}, 23:823--840, 2019.

\bibitem{chen2020lower}
Xinxin Chen and Hui He.
\newblock Lower deviation and moderate deviation probabilities for maximum of a branching random walk.
\newblock {\em Annales de l'Institut Henri Poincar\'{e} - Probabilit\'{e}s et Statistiques}, 56(4):2507--2539, 2023.

\bibitem{chen2020branching}
Xinxin Chen, Hui He, and Bastien Mallein.
\newblock Branching {B}rownian motion conditioned on small maximum.
\newblock {\em ALEA, Lat. Am. J. Probab. Math. Stat.}, 20:905--940, 2023.

\bibitem{conroy2022efficient}
Michael Conroy and Mariana Olvera-Cravioto.
\newblock Efficient rare event estimation for maxima of branching random walks.
\newblock In {\em 2022 Winter Simulation Conference (WSC)}, pages 145--155. IEEE, 2022.

\bibitem{dembo2009large}
Amir Dembo and Ofer Zeitouni.
\newblock {\em Large Deviations Techniques and Applications}.
\newblock Springer, 1998.

\bibitem{gantert2018large}
Nina Gantert and Thomas H{\"o}felsauer.
\newblock Large deviations for the maximum of a branching random walk.
\newblock {\em Electronic Communications in Probability}, 23(34):1--12, 2018.

\bibitem{harris2017many}
Simon~C Harris and Matthew~I Roberts.
\newblock The many-to-few lemma and multiple spines.
\newblock {\em Annales de l'Institut Henri Poincar\'{e} - Probabilit\'{e}s et Statistiques}, 53(1):226--242, 2017.

\bibitem{luo2025precise}
Lianghui Luo.
\newblock Precise upper deviation estimates for the maximum of a branching random walk.
\newblock {\em Electronic Journal of Probability}, 30:1--30, 2025.

\bibitem{lyons1997simple}
Russell Lyons.
\newblock A simple path to {B}iggins’ martingale convergence for branching random walk.
\newblock {\em Classical and Modern Branching Processes}, pages 217--221, 1997.

\bibitem{mallein2015maximal}
Bastien Mallein.
\newblock Maximal displacement in the {$d$}-dimensional branching {B}rownian motion.
\newblock {\em Electronic Communications in Probability}, 20:1--12, 2015.

\bibitem{mohanty2025strength}
Shaswat Mohanty, Jose Blanchet, Zhigang Suo, and Wei Cai.
\newblock Why is the strength of a polymer network so low?
\newblock {\em arXiv preprint arXiv:2502.11339}, 2025.

\bibitem{stone1967local}
Charles Stone.
\newblock On local and ratio limit theorems.
\newblock In {\em Proc. Fifth Berkeley Sympos. Math. Statist. and Probability (Berkeley, Calif., 1965/66)}, volume~2, pages 217--224. Univ. California Press Berkeley, Calif, 1967.

\bibitem{yin2020topological}
Yikai Yin, Nicolas Bertin, Yanming Wang, Zhenan Bao, and Wei Cai.
\newblock Topological origin of strain induced damage of multi-network elastomers by bond breaking.
\newblock {\em Extreme Mechanics Letters}, 40:100883, 2020.

\bibitem{yin2024network}
Yikai Yin, Shaswat Mohanty, Christopher~B Cooper, Zhenan Bao, and Wei Cai.
\newblock Network evolution controlling strain-induced damage and self-healing of elastomers with dynamic bonds.
\newblock {\em Macromolecules}, 57(13):6410--6418, 2024.

\bibitem{zeitouni2016branching}
Ofer Zeitouni.
\newblock Branching random walks and {G}aussian fields.
\newblock {\em Probability and Statistical Physics in St. Petersburg}, 91:437--471, 2016.

\bibitem{zhang2024modeling}
Zhenyuan Zhang, Shaswat Mohanty, Jose Blanchet, and Wei Cai.
\newblock Modeling shortest paths in polymeric networks using spatial branching processes.
\newblock {\em Journal of the Mechanics and Physics of Solids}, page 105636, 2024.

\end{thebibliography}

\appendix

\appendix
\section{Index of frequently used notation}\label{appendix:notation}

\begin{tabularx}{\textwidth}{@{}lX@{}}

\toprule

 $d$ & Underlying dimension of the BRW, $d\geq 1$\\
  $\zeta$ & Reproduction law of the BRW; $\p(\zeta=j)=p_j$\\
$\bxi$ & Jump distribution of the BRW\\
$\phi_\bxi$ & Moment generating function of $\bxi$\\
$S$ & Event of survival at all times\\
$p,q$ & $p=\p(S)=1-q$\\
  $\gamma$& $-\log\E[\zeta q^{\zeta-1}]$\\
$I(\bx)$&Large deviation rate function of $\bxi$\\
$\rho$ & Expected number of particles in the first generation, $\rho=\E[\zeta]=\sum_{j\geq 1}jp_j$\\
$c_1$ & Defined through $I(c_1,\z)=\log\rho$\\
$\hat{c}_1$ & The hypothetical speed of BRW in large deviation regimes (e.g.~$\hat{c}_1>c_1$ means lower deviation for FPT)\\
$\hat{\bc}_1$ & $(\hat{c}_1,\z)$\\
$\hat{\bc}_2$&$\nabla I(\hat{c}_1,\z)$\\
$B_x$ & Unit ball centered at $\bx=(x,0,\dots,0)\in\R^d$\\
$B_\bz$ & Unit ball centered at $\bz\in\R^d$\\
$B_r(\bz)$ & Ball centered at $\bz\in\R^d$ of radius $r$\\
$\tau_x$  & First passage time of $d$-dimensional BRW to $B_x$ \\
$M_n$&Maximum among the first coordinates of all particles' locations at time $n$\\
 $V_n$& Collection of particles at time $n$\\
$\bet_v$ & Location of the particle $v$\\
$\q$ & Law of the spine BRW\\
$\{\bS_k\}$ & Spine random walk\\
$\{\bb_k(i)\}$ & Adjacent jumps (bones) attached to the spine\\
$b_k(i)$ & The first coordinate of $\bb_k(i)$\\
$\{\bet^{(i,k)}\}$ & Copies of i.i.d.~BRW attached to the bones\\
$V^{(i,k)}_j$ & Collection of particles that left the spine at time $k$ and joined bone $i$, at time $k+1+j$\\ 
$V_{n,k}$&Particles at time $n$ that left the spine at time $k$\\
$\bar{V}_{n,k}$&All past locations of particles in $V_{n,k}$\\
$D_{n,k}$&Point process associated to $V_{n,k}$\\
$\bar{D}_{n,k}$&Point process associated to $\bar{V}_{n,k}$\\
$W_n$& A subset of $V_n$ that left the spine late; see \eqref{eq:Wn}\\
\bottomrule
\label{tab:TableOfNotationForMyResearch}
\end{tabularx}

\section{Deferred proofs}\label{app:proofs}
\begin{proof}[Proof of Lemma \ref{lemma:S_k}]

(i) By the definition of $\{\bb_k(i)\}$, conditioned on having $k$ offsprings from a spine particle, the law of their displacements (with respect to their parent on the spine) under $\q^{(k)}$ satisfies
\begin{align}
    \frac{\d\q^{(k)}}{\d\p_\bxi^{\otimes k}}(\bx_1,\dots,\bx_k)\propto \sum_{i=1}^k e^{\bth\cdot\bx_i}.\label{eq:p1}
\end{align}
Furthermore, the next spine particle is selected with probability proportional to $e^{\bth\cdot\bxi_i}$. 
It follows from \eqref{eq:p1} that for any Borel set $A$,
\begin{align*}
    \q_{\mathrm{spine}}(A)&=\int \frac{\sum_{i=1}^k\bone_{\{\bx_i\in A\}}e^{\bth\cdot\bx_i}}{\sum_{i=1}^ke^{\bth\cdot\bx_i}}\d\q^{(k)}(\bx_1,\dots,\bx_k)\\
    &\propto \int \sum_{i=1}^k\bone_{\{\bx_i\in A\}}e^{\bth\cdot\bx_i}\d\p_\bxi^{\otimes k}(\bx_1,\dots,\bx_k)\\
    &\propto \int \bone_{\{\bx\in A\}}e^{\bth\cdot\bx}\d\p_\bxi(\bx).
\end{align*}
The non-lattice property directly follows from the definition (assumption (A4)). This proves (i).

(ii) As in our setup $\bth=\hat{\bc}_2$ and $\nabla I(\hat{c}_1,\z)=\hat{\bc}_2$, the mean is $(\hat{c}_1,\z)$. The finiteness of the variance follows from the existence of exponential moments in a neighborhood of $\hat{\bc}_2$. 
 
  (iii) This follows from the local CLT (see e.g.~Corollary 1 of \citep{stone1967local}).
\end{proof}

\begin{proof}[Proof of Lemma \ref{lemma:b_k}]
    We follow the same lines as the proof of Lemma 2.4 of \citep{luo2025precise}. First, note that (A1) and (A3) together imply the following integrability condition:
\begin{align}
    \E\bigg[\Big(\sum_{u\in V_1}e^{\hat{\bc}_2\cdot\bet_u}\Big)^{1+\delta}\bigg]<\infty.\label{eq:1+delta condition}
\end{align}
Indeed, by first applying H\"{o}lder's inequality then conditioning on $|V_1|$,
\begin{align*}
     \E\bigg[\Big(\sum_{u\in V_1}e^{\hat{\bc}_2\cdot\bet_u}\Big)^{1+\delta}\bigg]&\leq  \E\bigg[\Big(\sum_{u\in V_1}e^{(1+\delta)\hat{\bc}_2\cdot\bet_u}\Big)|V_1|^\delta\bigg]\\
     &=\sum_{i=1}^\infty p_i i^{1+\delta}\E\big[e^{(1+\delta)\hat{\bc}_2\cdot\bet_u}\big]<\infty.
\end{align*}
By \eqref{eq:1+delta condition} and the definition of $\{\bb_\ell(i)\}$, 
    $$\E^\q\bigg[\Big(\sum_{i=1}^{N_\ell} e^{\hat{\bc}_2\cdot\bb_\ell(i)}\Big)^{\delta}\bigg]\ll\E\bigg[\Big(\sum_{u\in V_1}e^{\hat{\bc}_2\cdot\bet_u}\Big)^{1+\delta}\bigg]<\infty.$$
It follows from Markov's inequality that
\begin{align*}
    \sum_{\ell\geq \ell_0}\q\bigg(\sum_{i=1}^{N_\ell} e^{\hat{\bc}_2 \cdot\bb_{\ell}(i)}\geq e^{\ee\ell }\bigg)&\ll \sum_{\ell\geq \ell_0} e^{-\ee\delta\ell  }\ll e^{-\ee\delta \ell_0 }.
\end{align*}
   Putting $\delta_1=\ee\delta $ completes the proof.
\end{proof}

\end{document}